\pdfoutput=1
\documentclass[11pt]{article}
\usepackage[margin=1in]{geometry}
\usepackage{amsmath}
\usepackage{amssymb}
\usepackage{enumitem}
\usepackage{amsthm}
\usepackage{float}
\usepackage{mathtools}
\usepackage{microtype}
\usepackage[longnamesfirst]{natbib}
\usepackage{xcolor}
\usepackage{subcaption}
\usepackage[hang,flushmargin]{footmisc}
\usepackage{tikz}
\usepackage{pgfplots}
\usepackage{booktabs}
\usepackage[hypertexnames=false]{hyperref}

\hypersetup{colorlinks=true,linkcolor=blue,urlcolor=blue,citecolor=blue}
\pgfplotsset{compat=1.18}
\newtoggle{journal}
\togglefalse{journal}

\renewcommand{\cite}{\citet}

\newtheorem{theorem}{Theorem}
\newtheorem{lemma}{Lemma}

\newtheorem{example}{Example}

\newtheorem{definition}{Definition}

\DeclareMathOperator{\Var}{Var}

\renewcommand{\P}{\ensuremath{\mathbb{P}}}
\newcommand{\E}{\ensuremath{\mathbb{E}}}
\newcommand{\T}{\ensuremath{\mathsf{T}}}
\newcommand{\R}{\ensuremath{\mathbb{R}}}
\newcommand{\I}{\ensuremath{\mathbb{I}}}
\newcommand{\N}{\ensuremath{\mathbb{N}}}

\newcommand{\cP}{\ensuremath{\mathcal{P}}}
\newcommand{\cN}{\ensuremath{\mathcal{N}}}

\newcommand{\cU}{\ensuremath{\mathcal{U}}}
\newcommand{\up}{\ensuremath{\mathrm{up}}}
\newcommand{\lo}{\ensuremath{\mathrm{lo}}}
\newcommand{\ind}{\ensuremath{\mathrm{ind}}}
\newcommand{\FHU}{\ensuremath{\mathrm{FHU}}}
\newcommand{\FHL}{\ensuremath{\mathrm{FHL}}}
\newcommand{\conv}{\ensuremath{{\hspace*{0.2mm}\mathrm{c}}}}

\newcommand{\diffi}{\,\mathrm{d}}

\newlist{inlineroman}{enumerate*}{1}
\setlist[inlineroman]{afterlabel=~,label=(\roman*)}

\newcommand{\figurescale}{\iftoggle{journal}{1}{0.91}}
 
\makeatletter%
\begin{document}

\title{
  Sharp Anti-Concentration Inequalities \\
  for Extremum Statistics via Copulas
}

\author{
  Matias D.\ Cattaneo\textsuperscript{1}
  \and Ricardo P.\ Masini\textsuperscript{2*}
  \and William G.\ Underwood\textsuperscript{3}
}

\maketitle

\footnotetext[1]{
  Department of Operations Research
  and Financial Engineering,
  Princeton University.
}
\footnotetext[2]{
  Department of Statistics,
  University of California, Davis.
}
\footnotetext[3]{
  Statistical Laboratory,
  University of Cambridge.
}
\let\thefootnote\relax
\footnotetext[1]{
  \textsuperscript{*}Corresponding author:
  \href{mailto:rmasini@ucdavis.edu}{\texttt{rmasini@ucdavis.edu}}
}
\newcommand{\thefootnote}{\arabic{footnote}}

\setcounter{page}{0}\thispagestyle{empty}

\begin{abstract}
  {\normalsize %
We derive sharp upper and lower bounds for the pointwise concentration function
of the maximum statistic of $d$ identically distributed
real-valued random variables. Our
first main result places no restrictions either on the common marginal law of
the samples or on the copula describing their joint distribution. We show that,
in general, strictly sublinear dependence of the concentration function on the
dimension $d$ is not possible.
We then introduce a new class of copulas, namely those with a convex diagonal
section, and demonstrate that restricting to this class yields a sharper upper
bound on the concentration function. This allows us to establish several new
dimension-independent and poly-logarithmic-in-$d$ anti-concentration
inequalities for a
variety of marginal distributions under mild dependence assumptions. Our
theory improves upon the best known results in certain special cases.
Applications to high-dimensional statistical inference are presented, including
a specific example pertaining to Gaussian mixture approximations for factor
models, for which our main results lead to superior distributional guarantees.
}
\end{abstract}

\vspace*{10mm}
\noindent\textbf{Keywords}:
Anti-concentration;
copulas;
high-dimensional probability;
concentration;
extreme value theory;
order statistics.

\vspace*{4mm}
\noindent\textbf{MSC}:
Primary
60E15; %
Secondary
62H05, %
62G32. %

\clearpage
\pagebreak

\tableofcontents
\pagebreak

\section{Introduction}

Concentration of measure has been extensively studied throughout the
probability and statistics literature. Anti-concentration phenomena, on the
other hand, appear much less frequently and are generally not so well
understood \citep{vershynin2007}.
While it is impossible to pin down the date when anti-concentration
became a topic of interest, its systematic study is commonly attributed to
\citet{levy1954}, who defined the concentration function of a real-valued
random variable $Y$ as
$L(Y, \varepsilon) \vcentcolon=
\sup_{x \in \R} \P(x \leq Y \leq x + \varepsilon)$
for $\varepsilon \geq 0$.
The early focus was almost exclusively on the asymptotic behavior of
the concentration function as $\varepsilon \to 0$, motivated by
applications to quantitative central limit theorems.
The last two decades have seen a revival
of interest in anti-concentration, fueled by advances in
high-dimensional and nonparametric statistics
\citep{bakshi2020outlier,chernozhukov2013gaussian,%
  chernozhukov2014anti,deng2020beyond,%
koike2021notes,kuchibhotla2021high},
random matrix theory
\citep{litvak2017adjacency,nie2022matrix},
geometric analysis
\citep{livshyts2014maximal,livshyts2021some,%
paouris2012small,paouris2018gaussian}
and applied probability
\citep{aizenman2009bernoulli,belloni2024anti,chernozhukov2015comparison,%
  fox2021combinatorial,gotze2017large,%
  kozbur2021dimension,%
  krishnapur2016anti,%
meka2015anticoncentration,rudelson2015small}.
Recently, attention has shifted to finding
sharp non-asymptotic upper bounds
for the concentration function in terms of $\varepsilon$
and properties of the law of $Y$.
One particular example of interest is the maximum statistic
$Y \vcentcolon= \max_{i \in [d]} X_i$, with
$X_1, \ldots, X_d$ real-valued random variables.

When the distribution of $Y$ admits a density $f(x)$ with
respect to the Lebesgue measure,
a simple upper bound for the concentration function
is obtained by observing that
$L(Y,\varepsilon) \leq \varepsilon \sup_{x \in \R} f(x)$.
This technique was applied by \cite[Theorem 3]{chernozhukov2015comparison}
to $Y \vcentcolon = \max_{i \in [d]} X_i$
with $(X_1,\dots, X_d)$ a zero-mean multivariate Gaussian random vector
with a non-singular covariance matrix.
Their proof leveraged the fact that conditioning on components
preserves joint Gaussianity,
and the resulting anti-concentration inequality was used to establish
a conditional multiplier central limit theorem in a high-dimensional regime.
A related approach is to provide bounds for the concentration function in
terms of the variance of $Y$; \cite{bobkov2015} used this method to
establish matching upper and lower bounds (up to a constant factor)
under a log-concavity assumption.
Unfortunately, if $X_i$ are log-concave random variables, then
there is no guarantee that $\max_{i \in [d]} X_i$ is similarly
log-concave (refer to \cite{saumard_wellner} for a comprehensive
review of log-concavity properties). Furthermore, lower bounds on
the variance of $\max_{i \in [d]} X_i$ are typically not easy to obtain unless
the joint distribution of $(X_1, \dots, X_d)$ is specified.
In the multivariate Gaussian setting,
\cite{giessing2023} recently established such bounds
in terms of the dimension or metric entropy of the joint distribution.
Another approach builds upon the seminal paper of
\cite{nazarov2003maximal}, establishing anti-concentration inequalities
for the maximum statistic
using properties of the Gaussian distribution
and tools from convex geometry
\citep{chernozhukov2017central,chernozhukov2017detailed}.
See also \cite[Theorem~10]{deng2020beyond}
for a refined inequality.

Our goal is to study the anti-concentration behavior of maximum statistics
by providing upper and lower bounds for the pointwise concentration function
\begin{align}
  \label{eq:pointwise_concentration}
  \P \biggl( x < \max_{i \in [d]} X_i \leq x + \varepsilon \biggr),
\end{align}
where $X_1, \ldots, X_d$ are real-valued random variables,
$x \in \R$ and $\varepsilon \geq 0$.
In contrast to several prior results,
we refrain from taking a supremum over $x \in \R$,
with our main results focusing on pointwise
(rather than uniform) anti-concentration phenomena.
In principle, this can lead to sharper inequalities when
restricting to $x$ lying in a subset of $\R$
(see Section~\ref{sec:application} for an illustration).
Further, we seek to impose minimal assumptions on the dependence
structure of the random vector $(X_1, \ldots, X_d)$,
as determined by its associated copula
\citep[see][for a contemporary review]{durante2016principles}.
We assume throughout that the variables
$X_1, \ldots, X_d$ share a common marginal distribution.

Our first main result, given as
Theorem~\ref{thm:common} in Section~\ref{sec:arbitrary},
gives upper and lower bounds for the pointwise concentration function
of $\max_{i \in [d]} X_i$, as defined in \eqref{eq:pointwise_concentration}.
Crucially, this theorem makes no assumptions at all on the copula
describing the dependence structure of $(X_1, \ldots, X_d)$.
As such, it is applicable even in cases
where the joint distribution is intractable or unspecified.
Moreover, we construct copulas which exactly attain our
upper and lower bounds, respectively;
therefore Theorem~\ref{thm:common} is not improvable
unless extra conditions are imposed on the copula.
When considering marginally Gaussian random variables
(Example~\ref{ex:gaussian_linear}), we show that
the worst-case concentration function
(i.e.,\ the maximum over all possible copulas)
is substantially larger (as a function of the dimension $d$) than when
assuming joint Gaussianity \citep[see][Theorem~1]{chernozhukov2017detailed}.
It is therefore essential in applications,
particularly in high-dimensional regimes,
to consider properties of the copula associated with
$(X_1, \ldots, X_d)$ as well as their marginal laws.
The proof of Theorem~\ref{thm:common},
presented in Section~\ref{sec:proof_thm_common}, relies only on basic
properties of copulas
and their diagonal sections.
A similar copula-based approach was taken by
\cite{frank1987best}, who obtained
optimal upper and lower bounds for the distribution function of the sum
(and other combinations) of several random variables,
under arbitrary dependence.

In Section~\ref{sec:convex} we obtain a more refined result
as Theorem~\ref{thm:convex} by restricting the class of copulas
under consideration.
Specifically, we impose a convexity condition on the diagonal section
of the copula; this assumption is novel, to the best of the authors' knowledge,
and leads to a class of copulas that could be useful in other applications.
We present an explicit copula for which our concentration function upper bound
is tight, demonstrating its optimality.
The resulting anti-concentration inequality
for the maximum statistic is typically substantially
stronger than that obtained using Theorem~\ref{thm:common};
when applied to a joint distribution with
Gaussian margins, we improve several well-known results in
the literature where
previously a multivariate Gaussian law was assumed
\citep[cf.][]{chernozhukov2015comparison,
chernozhukov2017detailed}.
Moreover, we demonstrate the applicability of
Theorem~\ref{thm:convex} to several popular families of copulas,
and discuss the resulting concentration bounds for
a variety of marginal distributions.

Section~\ref{sec:application} presents an application of our main
results in high-dimensional statistical inference,
highlighting the importance of sharp anti-concentration bounds
in distributional analysis.
We give an explicit example in the context of
Gaussian mixture approximations for high-dimensional factor models.
In particular, we demonstrate that our main results lead to superior
anti-concentration inequalities, and therefore better
guarantees on the quality of the distributional approximation,
especially when the Gaussian mixture components
exhibit a wide range of variances.

Proofs and further details are given in
Section~\ref{sec:proofs};
Section~\ref{sec:conclusion} contains
concluding remarks.

\subsection{Notation}

We use $\N \vcentcolon= \{1, 2, \ldots \}$ for the natural numbers,
and for $d \in \N$ we define $[d] \vcentcolon= \{1, \ldots, d\}$.
The multivariate normal distribution with mean vector
$\mu$ and covariance matrix $\Sigma$ is denoted by $\cN(\mu, \Sigma)$,
and the cumulative distribution function
(CDF) and Lebesgue density function of $\cN(0, 1)$
are written as $\Phi$ and $\phi$ respectively.
The uniform distribution on $[0, 1]$ is denoted by $\cU$.
For $a, b \in \R$, $a \land b$ and $a \lor b$
are their minimum and maximum, respectively.
For a function $F$ of a single real variable,
we use $F^-(x)$ for its left limit at $x$ if it exists.
For two functions $f$ and $g$, we write
$f \circ g(x) = f\bigl(g(x)\bigr)$ for
their composition whenever it is well-defined.
The natural logarithm is denoted by $\log$.

\subsection{Preliminary results}%

Suppose that $X_1, \ldots, X_d$ are real-valued random variables with
a common distribution function $F : \R \to [0,1]$.
By a well-known theorem due to Sklar
\citep[Theorem~2.10.9]{nelsen2006},
the joint law of $(X_1, \ldots, X_d)$ decomposes as
$\P ( X_1 \leq x_1, \ldots, X_d \leq x_d )
= C\bigl(F(x_1), \ldots, F(x_d)\bigr)$,
where $C$ is a $d$-dimensional copula
\citep[Definition~2.10.6]{nelsen2006}.
Considering $x_1 = \cdots = x_d =\vcentcolon x$, we obtain
\begin{align}
  \label{eq:copula_diagonal}
  \P \biggl( x < \max_{i \in [d]} X_i \leq x + \varepsilon \biggr)
  &= C\bigl(F(x + \varepsilon), \ldots, F(x + \varepsilon)\bigr)
  - C\bigl(F(x), \ldots, F(x)\bigr).
\end{align}
As such, the distribution of $\max_{i \in [d]} X_i$
depends on the copula associated with
$(X_1, \ldots, X_d)$ only through its
diagonal section, as formalized in Definition~\ref{def:diagonal}.

\begin{definition}
  \label{def:diagonal}
  Let $d \in \N$. A function
  $\Delta : [0, 1] \to [0, 1]$ is a
  \emph{$d$-dimensional copula diagonal}
  if there exists a $d$-dimensional copula
  $C: [0, 1]^d \to [0, 1]$ with
  $\Delta(u) = C(u, \ldots, u)$ for all $u \in [0, 1]$.
\end{definition}

Lemma~\ref{lem:copula_diagonal} below
gives a characterization of $d$-dimensional
copula diagonals.
In \cite{fernandez2018constructions},
an explicit copula $C$
is constructed with a specified diagonal $\Delta$;
for our purposes, any such copula suffices by \eqref{eq:copula_diagonal}.
See \cite{cuculescu2001copulas} and \cite{jaworski2009copulas}
for further background on copulas and their diagonals.

\begin{lemma}[Theorem~1, \citealp{fernandez2018constructions}]
  \label{lem:copula_diagonal}
  A function $\Delta : [0, 1] \to [0, 1]$ is a
  $d$-dimensional copula diagonal if and only if it satisfies:
  \begin{inlineroman}
    \item $\Delta(1) = 1$;
    \item $\Delta(u) \leq u$ for all $u \in [0, 1]$; and
    \item $0 \leq \Delta(u') - \Delta(u) \leq d(u'-u)$
      for all $u, u' \in [0, 1]$ with $u \leq u'$.
  \end{inlineroman}
\end{lemma}

\section{Anti-concentration inequalities for arbitrary copulas}
\label{sec:arbitrary}

We derive sharp upper and lower bounds on the
pointwise concentration
function of the maximum statistic of identically distributed
(not necessarily independent) random variables,
imposing no further assumptions on either their common marginal law
or the copula describing their joint distribution.
The relevant class of distributions is specified in
Definition~\ref{def:arbitrary}.

\begin{definition}
  \label{def:arbitrary}
  Let $d \in \N$ and $F: \R \to [0, 1]$ be a CDF.
  Write $\cP_d(F)$ for the set of distributions
  $\P$ on $\R^d$ which have joint CDFs of the form
  \begin{align*}
    \P \bigl( X_1 \leq x_1, \ldots, X_d \leq x_d \bigr)
    &= C\bigl(F(x_1), \ldots, F(x_d)\bigr),
  \end{align*}
  for some $d$-dimensional copula $C$.
\end{definition}

Our first main result is given in Theorem~\ref{thm:common}.

\begin{theorem}
  \label{thm:common}
  Let $d \in \N$ and $F: \R \to [0, 1]$ be a CDF.
  For each $x \in \R$ and $\varepsilon \geq 0$,
  \begin{align}
    \label{eq:common_upper}
    \max_{\P \in \cP_d(F)}
    \P \biggl( x < \max_{i \in [d]} X_i \leq x + \varepsilon \biggr)
    &=
    \bigl\{d \bigl(F(x + \varepsilon) - F(x)\bigr)\bigr\}
    \land F(x + \varepsilon), \\
    \label{eq:common_lower}
    \min_{\P \in \cP_d(F)}
    \P \biggl( x < \max_{i \in [d]} X_i \leq x + \varepsilon \biggr)
    &=
    0 \lor
    \bigl\{1 - F(x) - d \bigl(1 - F(x + \varepsilon)\bigr)\bigr\}.
  \end{align}
\end{theorem}

Equation \eqref{eq:common_upper} in Theorem~\ref{thm:common}
gives a tight upper bound on the
probability of the maximum statistic falling in
$(x, x + \varepsilon]$.
Further, \eqref{eq:common_upper} shows that if $F(x) \in (0, 1)$
and $F(x + \varepsilon) - F(x) \leq \frac{F(x)}{d-1}$,
then there exists a joint distribution
$\P_\up$
such that the maximum statistic exhibits
strong local concentration near $x$. That is,
$\P_\up \bigl( x < \max_{i \in [d]} X_i \leq x + \varepsilon \bigr)
= d \bigl(F(x + \varepsilon) - F(x)\bigr)$,
which increases linearly with the dimension $d$;
see \eqref{eq:diagonal_upper} in
Section~\ref{sec:proof_overview}.
If also $F$ admits a Lebesgue density $f$ on
$(x, x + \varepsilon]$ which is
bounded above by $M$ and below by $m$, then
$\varepsilon \leq \frac{F(x)}{M(d-1)}$ implies
$d m \varepsilon \leq
\P_\up \bigl( x < \max_{i \in [d]} X_i \leq x + \varepsilon \bigr)
\leq d M \varepsilon$.
This is a form of ``curse of dimensionality,''
precluding the possibility of obtaining anti-concentration
bounds which hold uniformly in $\varepsilon$ and
depend strictly sublinearly on $d$.
This local concentration phenomenon
can occur at any point $x \in \R$ satisfying $F(x) > 0$;
contrast this with the independent setting, in which
concentration is restricted to regions where $F(x)$ is close to $1$.

For a lower bound on the concentration probability,
\eqref{eq:common_lower} establishes a joint distribution $\P_\lo$
which achieves exact anti-concentration whenever
$F(x + \varepsilon) - F(x) \leq \frac{d-1}{d}\bigl(1-F(x)\bigr)$,
in the sense that
$\P_\lo \bigl( x < \max_{i \in [d]} X_i \leq x + \varepsilon \bigr) = 0$.
If $F$ admits a Lebesgue density $f$ on
$(x, x + \varepsilon]$ which is bounded above by $M$, then
$\varepsilon \leq \frac{d-1}{M d}\bigl(1-F(x)\bigr)$
suffices to ensure this;
see \eqref{eq:diagonal_lower} in Section~\ref{sec:proof_overview}.

In Example~\ref{ex:gaussian_linear} we apply the result from
\eqref{eq:common_upper} with marginally Gaussian random variables.

\begin{example}[Marginal Gaussian distribution]
  \label{ex:gaussian_linear}
  Let $d \in \N$, $\sigma > 0$ and $\varepsilon \in [0, \sigma]$.
  By \eqref{eq:common_upper}, there exists $(X_1, \ldots, X_d)$
  with $X_i \sim \mathcal N(0, \sigma^2)$ for $i \in [d]$ such that
  \begin{align*}
    \sup_{x \in \R}
    \P \biggl( x < \max_{i \in [d]} X_i \leq x + \varepsilon \biggr)
    &\geq
    \frac{d \varepsilon}{\sigma}
    \phi\biggl(\frac{\varepsilon}{\sigma}\biggr)
    \land \Phi\biggl(\frac{\varepsilon}{\sigma}\biggr)
    \geq
    \frac{d \varepsilon}{\sigma}
    \frac{e^{-1/2}}{\sqrt{2 \pi}}
    \land \frac{1}{2}
    \geq
    \frac{d \varepsilon}{5\sigma}
    \land \frac{1}{2}.
  \end{align*}
\end{example}

Compare Example~\ref{ex:gaussian_linear} with Nazarov's inequality
(\citealp{nazarov2003maximal};
see also \citealp{chernozhukov2017detailed}, for a detailed proof)
which, under the assumption of joint Gaussianity, obtains a bound of
\begin{align}
  \label{eq:nazarov}
  \sup_{x \in \R}
  \P \biggl( x < \max_{i \in [d]} X_i \leq x + \varepsilon \biggr)
  \leq
  \frac{\varepsilon}{\sigma}
  \Bigl(\sqrt{2 \log d} + 2\Bigr).
\end{align}
This slow-growing dependence on $d$ is crucial in
high-dimensional statistical applications, where the
dimension may be much larger than the sample size.
Example~\ref{ex:gaussian_linear} shows that marginal Gaussianity
of each $X_i$ alone is insufficient for obtaining such a bound.
Therefore, in the upcoming Section~\ref{sec:convex} we present
a restricted class of copulas for
which sharper anti-concentration inequalities hold
than those given in Theorem~\ref{thm:common}.
We recover a form of Nazarov's inequality
\eqref{eq:nazarov} as a special case.

\subsection{Overview of proof strategy}
\label{sec:proof_overview}

The proof of Theorem~\ref{thm:common} is presented
in Section~\ref{sec:proof_thm_common} and proceeds as follows.
Firstly, we consider the special case where the common law
of each variable is the standard uniform distribution,
and write $U_i$ instead of $X_i$ for clarity.
The joint CDF of
$(U_1, \ldots, U_d)$ is a $d$-dimensional copula $C: [0, 1]^d \to [0, 1]$.
With $\Delta(u) \vcentcolon= C(u, \ldots, u)$
the diagonal section of $C$, for $u \in [0, 1]$ and $\delta \in [0, 1-u]$,
\eqref{eq:copula_diagonal} gives
\begin{align}
  \label{eq:max_diagonal}
  \P \biggl( u < \max_{i \in [d]} U_i \leq u + \delta \biggr)
  = \Delta(u + \delta) - \Delta(u).
\end{align}
Establishing \eqref{eq:common_upper} and \eqref{eq:common_lower}
thus reduces to finding $\Delta_\up$ and $\Delta_\lo$
which maximize and minimize
the right-hand side of \eqref{eq:max_diagonal} respectively over $\Delta$,
subject to the constraints enforced in Lemma~\ref{lem:copula_diagonal}.
The resulting copula diagonals are described in
\eqref{eq:diagonal_upper} and \eqref{eq:diagonal_lower},
and are plotted in Figure~\ref{fig:copulas}.
Fix $d \in \N$ and $u \in [0, 1]$, and for $t \in [0, 1]$ define
\begin{align}
  \label{eq:diagonal_upper}
  \Delta_\up(t)
  &\vcentcolon=
  d \cdot \biggl\{t - \biggl(u \land \frac{d-1}{d}\biggr)\biggr\}
  \cdot \I \biggl\{ u \land \frac{d-1}{d} < t \leq \frac{d u}{d-1} \biggr\}
  + t \cdot \I \biggl\{ \frac{d u}{d-1} \land 1 < t \biggr\}, \\
  \label{eq:diagonal_lower}
  \Delta_\lo(t)
  &\vcentcolon=
  t \cdot \I \{ t \leq u \}
  + u \cdot \I \biggl\{ u < t \leq \frac{d+u-1}{d} \biggr\}
  + (1 - d + d \cdot t)
  \cdot \I \biggl\{ \frac{d+u-1}{d} < t \biggr\}.
\end{align}

For the upper bound
(\ref{eq:diagonal_upper}, Figure~\ref{fig:copula_upper}),
we maximize the increment of
$\Delta$ over $(u, u + \delta]$ to obtain $\Delta_\up$;
for the lower bound
(\ref{eq:diagonal_lower}, Figure~\ref{fig:copula_lower}),
we minimize it, yielding $\Delta_\lo$. Therefore,
\begin{align*}
  \P_\up \Bigl( u < \max_{i \in [d]} U_i \leq u + \delta \Bigr)
  &=
  (d \delta) \land (u + \delta), \\
  \P_\lo \Bigl( u < \max_{i \in [d]} U_i \leq u + \delta \Bigr)
  &=
  0 \lor
  \bigl(1 - u - d(1 - u - \delta)\bigr).
\end{align*}
The generalization to an arbitrary distribution function
$F$ then proceeds by a quantile transform, taking
$u = F(x)$ and $\delta = F(x + \varepsilon) - F(x)$,
and finally setting $X_i = F^{-1}(U_i)$.

Analogous results to those in
Theorem~\ref{thm:common} can be derived with the maximum statistic replaced by
the minimum statistic by considering the variables $-X_i$,
with common CDF $G(x) = 1 - F^-(-x)$.
If $F$ is symmetric in the sense that
$F(x) = 1-F^-(-x)$ for all $x \in \R$, then similar results also hold
for the maximum absolute value statistic
(see Example~\ref{ex:multivariate_gaussian}),
noting that $\max_{i \in [d]} |X_i| = \max_{i \in [d]} (X_i \lor -X_i)$
and applying Theorem~\ref{thm:common} to the $2d$-dimensional vector
$(X_1, \ldots, X_d, -X_1, \ldots, -X_d)$.

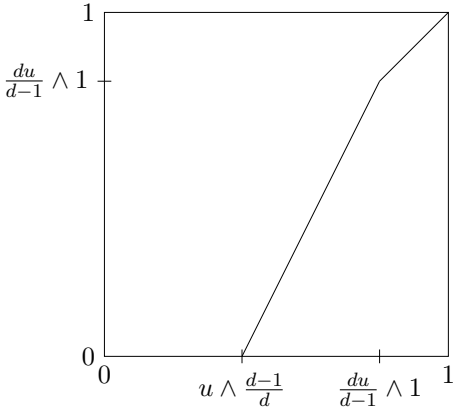
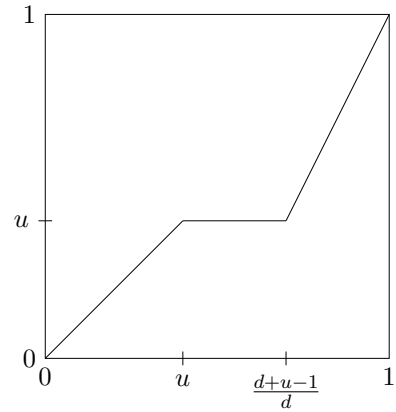
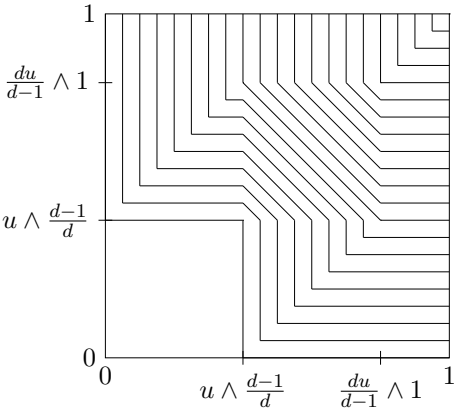
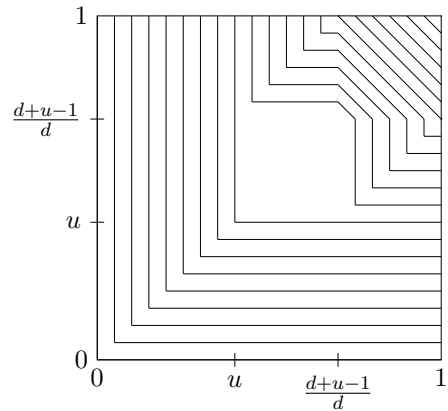
\begin{figure}[H]
  \centering
  \begin{subfigure}{0.49\textwidth}
    \centering
    \scalebox{\figurescale}{%
\begin{tikzpicture}
  \draw (0,0) rectangle (5,5);
  \draw (0,0) node[below] {$0$};
  \draw (5,0) node[below] {$1$};
  \draw (0,5) node[left] {$1$};
  \draw (0,0) node[left] {$0$};
  \draw[-] (2,0.1) -- (2,-0.1) node[below] {$u \land \frac{d-1}{d}$};
  \draw[-] (4,0.1) -- (4,-0.1) node[below] {$\frac{d u}{d-1} \land 1$};
  \draw[-] (0.1,4) -- (-0.1,4) node[left] {$\frac{d u}{d-1} \land 1$};
  \path (0.1,4) -- (-0.1,4) node[left] {\phantom{$\frac{d u}{d-1} \land 1$}};
  \draw (2,0) -- (4,4);
  \draw (4,4) -- (5,5);
\end{tikzpicture}
}
    \caption{The copula diagonal $\Delta_\up$
    defined in \eqref{eq:diagonal_upper}.}
    \label{fig:copula_upper}
    \vspace*{3mm}
  \end{subfigure}
  \begin{subfigure}{0.49\textwidth}
    \centering
    \scalebox{\figurescale}{%
\begin{tikzpicture}
  \draw (0,0) rectangle (5,5);
  \draw (0,0) node[below] {$0$};
  \draw (5,0) node[below] {$1$};
  \draw (0,5) node[left] {$1$};
  \draw (0,0) node[left] {$0$};
  \draw[-] (2,0.1) -- (2,-0.1) node[below] {$u$};
  \draw[-] (3.5,0.1) -- (3.5,-0.1) node[below] {$\frac{d+u-1}{d}$};
  \draw[-] (0.1,2) -- (-0.1,2) node[left] {$u$};
  \path (0.1,3.5) -- (-0.1,3.5) node[left] {\phantom{$\frac{d+u-1}{d}$}};
  \draw (0,0) -- (2,2);
  \draw (2,2) -- (3.5,2);
  \draw (3.5,2) -- (5,5);
\end{tikzpicture}
}
    \caption{The copula diagonal $\Delta_\lo$
    defined in \eqref{eq:diagonal_lower}.}
    \label{fig:copula_lower}
    \vspace*{3mm}
  \end{subfigure}
  \begin{subfigure}{0.49\textwidth}
    \centering
    \scalebox{\figurescale}{%

\begin{tikzpicture}
  \draw (0,0) rectangle (5,5);
  \draw (0,0) node[below] {$0$};
  \draw (5,0) node[below] {$1$};
  \draw (0,5) node[left] {$1$};
  \draw (0,0) node[left] {$0$};
  \draw[-] (2.0,0.1) -- (2.0,-0.1)
  node[below] {$u \land \frac{d-1}{d}$};
  \draw[-] (4.0,0.1) -- (4.0,-0.1)
  node[below] {$\frac{d u}{d-1} \land 1$};
  \draw[-] (0.1,2.0) -- (-0.1,2.0)
  node[left] {$u \land \frac{d-1}{d}$};
  \draw[-] (0.1,4.0) -- (-0.1,4.0)
  node[left] {$\frac{d u}{d-1} \land 1$};
  \draw[-] (0.0,5) -- (0.0,2.0);
  \draw[-] (0.0,2.0) -- (2.0,2.0);
  \draw[-] (2.0,2.0) -- (2.0,2.0);
  \draw[-] (2.0,2.0) -- (2.0,0.0);
  \draw[-] (2.0,0.0) -- (5,0.0);
  \draw[-] (0.25,5) -- (0.25,2.25);
  \draw[-] (0.25,2.25) -- (2.0,2.25);
  \draw[-] (2.0,2.25) -- (2.25,2.0);
  \draw[-] (2.25,2.0) -- (2.25,0.25);
  \draw[-] (2.25,0.25) -- (5,0.25);
  \draw[-] (0.5,5) -- (0.5,2.5);
  \draw[-] (0.5,2.5) -- (2.0,2.5);
  \draw[-] (2.0,2.5) -- (2.5,2.0);
  \draw[-] (2.5,2.0) -- (2.5,0.5);
  \draw[-] (2.5,0.5) -- (5,0.5);
  \draw[-] (0.75,5) -- (0.75,2.75);
  \draw[-] (0.75,2.75) -- (2.0,2.75);
  \draw[-] (2.0,2.75) -- (2.75,2.0);
  \draw[-] (2.75,2.0) -- (2.75,0.75);
  \draw[-] (2.75,0.75) -- (5,0.75);
  \draw[-] (1.0,5) -- (1.0,3.0);
  \draw[-] (1.0,3.0) -- (2.0,3.0);
  \draw[-] (2.0,3.0) -- (3.0,2.0);
  \draw[-] (3.0,2.0) -- (3.0,1.0);
  \draw[-] (3.0,1.0) -- (5,1.0);
  \draw[-] (1.25,5) -- (1.25,3.25);
  \draw[-] (1.25,3.25) -- (2.0,3.25);
  \draw[-] (2.0,3.25) -- (3.25,2.0);
  \draw[-] (3.25,2.0) -- (3.25,1.25);
  \draw[-] (3.25,1.25) -- (5,1.25);
  \draw[-] (1.5,5) -- (1.5,3.5);
  \draw[-] (1.5,3.5) -- (2.0,3.5);
  \draw[-] (2.0,3.5) -- (3.5,2.0);
  \draw[-] (3.5,2.0) -- (3.5,1.5);
  \draw[-] (3.5,1.5) -- (5,1.5);
  \draw[-] (1.75,5) -- (1.75,3.75);
  \draw[-] (1.75,3.75) -- (2.0,3.75);
  \draw[-] (2.0,3.75) -- (3.75,2.0);
  \draw[-] (3.75,2.0) -- (3.75,1.75);
  \draw[-] (3.75,1.75) -- (5,1.75);
  \draw[-] (2.0,5) -- (2.0,4.0);
  \draw[-] (2.0,4.0) -- (2.0,4.0);
  \draw[-] (2.0,4.0) -- (4.0,2.0);
  \draw[-] (4.0,2.0) -- (4.0,2.0);
  \draw[-] (4.0,2.0) -- (5,2.0);
  \draw[-] (2.25,5) -- (2.25,4.0);
  \draw[-] (2.25,4.0) -- (4.0,2.25);
  \draw[-] (4.0,2.25) -- (5,2.25);
  \draw[-] (2.5,5) -- (2.5,4.0);
  \draw[-] (2.5,4.0) -- (4.0,2.5);
  \draw[-] (4.0,2.5) -- (5,2.5);
  \draw[-] (2.75,5) -- (2.75,4.0);
  \draw[-] (2.75,4.0) -- (4.0,2.75);
  \draw[-] (4.0,2.75) -- (5,2.75);
  \draw[-] (3.0,5) -- (3.0,4.0);
  \draw[-] (3.0,4.0) -- (4.0,3.0);
  \draw[-] (4.0,3.0) -- (5,3.0);
  \draw[-] (3.25,5) -- (3.25,4.0);
  \draw[-] (3.25,4.0) -- (4.0,3.25);
  \draw[-] (4.0,3.25) -- (5,3.25);
  \draw[-] (3.5,5) -- (3.5,4.0);
  \draw[-] (3.5,4.0) -- (4.0,3.5);
  \draw[-] (4.0,3.5) -- (5,3.5);
  \draw[-] (3.75,5) -- (3.75,4.0);
  \draw[-] (3.75,4.0) -- (4.0,3.75);
  \draw[-] (4.0,3.75) -- (5,3.75);
  \draw[-] (4.0,5) -- (4.0,4.0);
  \draw[-] (4.0,4.0) -- (4.0,4.0);
  \draw[-] (4.0,4.0) -- (5,4.0);
  \draw[-] (4.25,5) -- (4.25,4.25);
  \draw[-] (4.25,4.25) -- (5,4.25);
  \draw[-] (4.5,5) -- (4.5,4.5);
  \draw[-] (4.5,4.5) -- (5,4.5);
  \draw[-] (4.75,5) -- (4.75,4.75);
  \draw[-] (4.75,4.75) -- (5,4.75);
  \draw[-] (5.0,5) -- (5.0,5.0);
  \draw[-] (5.0,5.0) -- (5,5.0);
\end{tikzpicture}
}
    \caption{Extension of $\Delta_\up$ to a copula with $d=2$.}
    \label{fig:copula_2d_upper}
  \end{subfigure}
  \begin{subfigure}{0.49\textwidth}
    \centering
    \scalebox{\figurescale}{%

\begin{tikzpicture}
  \draw (0,0) rectangle (5,5);
  \draw (0,0) node[below] {$0$};
  \draw (5,0) node[below] {$1$};
  \draw (0,5) node[left] {$1$};
  \draw (0,0) node[left] {$0$};
  \draw[-] (2.0,0.1) -- (2.0,-0.1) node[below] {$u$};
  \draw[-] (0.1,2.0) -- (-0.1,2.0) node[left] {$u$};
  \draw[-] (3.5,0.1) -- (3.5,-0.1)
  node[below] {$\frac{d+u-1}{d}$};
  \draw[-] (0.1,3.5) -- (-0.1,3.5)
  node[left] {$\frac{d+u-1}{d}$};
  \draw[-] (0.0,5) -- (0.0,0.0);
  \draw[-] (0.0,0.0) -- (5,0.0);
  \draw[-] (0.25,5) -- (0.25,0.25);
  \draw[-] (0.25,0.25) -- (5,0.25);
  \draw[-] (0.5,5) -- (0.5,0.5);
  \draw[-] (0.5,0.5) -- (5,0.5);
  \draw[-] (0.75,5) -- (0.75,0.75);
  \draw[-] (0.75,0.75) -- (5,0.75);
  \draw[-] (1.0,5) -- (1.0,1.0);
  \draw[-] (1.0,1.0) -- (5,1.0);
  \draw[-] (1.25,5) -- (1.25,1.25);
  \draw[-] (1.25,1.25) -- (5,1.25);
  \draw[-] (1.5,5) -- (1.5,1.5);
  \draw[-] (1.5,1.5) -- (5,1.5);
  \draw[-] (1.75,5) -- (1.75,1.75);
  \draw[-] (1.75,1.75) -- (5,1.75);
  \draw[-] (2.0,5) -- (2.0,2.0);
  \draw[-] (2.0,2.0) -- (5,2.0);
  \draw[-] (2.25,5) -- (2.25,3.75);
  \draw[-] (2.25,3.75) --(3.5,3.75);
  \draw[-] (3.5,3.75) --(3.75,3.5);
  \draw[-] (3.75,3.5) --(3.75,2.25);
  \draw[-] (3.75,2.25) -- (5,2.25);
  \draw[-] (2.5,5) -- (2.5,4.0);
  \draw[-] (2.5,4.0) --(3.5,4.0);
  \draw[-] (3.5,4.0) --(4.0,3.5);
  \draw[-] (4.0,3.5) --(4.0,2.5);
  \draw[-] (4.0,2.5) -- (5,2.5);
  \draw[-] (2.75,5) -- (2.75,4.25);
  \draw[-] (2.75,4.25) --(3.5,4.25);
  \draw[-] (3.5,4.25) --(4.25,3.5);
  \draw[-] (4.25,3.5) --(4.25,2.75);
  \draw[-] (4.25,2.75) -- (5,2.75);
  \draw[-] (3.0,5) -- (3.0,4.5);
  \draw[-] (3.0,4.5) --(3.5,4.5);
  \draw[-] (3.5,4.5) --(4.5,3.5);
  \draw[-] (4.5,3.5) --(4.5,3.0);
  \draw[-] (4.5,3.0) -- (5,3.0);
  \draw[-] (3.25,5) -- (3.25,4.75);
  \draw[-] (3.25,4.75) --(3.5,4.75);
  \draw[-] (3.5,4.75) --(4.75,3.5);
  \draw[-] (4.75,3.5) --(4.75,3.25);
  \draw[-] (4.75,3.25) -- (5,3.25);
  \draw[-] (3.5,5) -- (3.5,5.0);
  \draw[-] (3.5,5.0) --(3.5,5.0);
  \draw[-] (3.5,5.0) --(5.0,3.5);
  \draw[-] (5.0,3.5) --(5.0,3.5);
  \draw[-] (5.0,3.5) -- (5,3.5);
  \draw[-] (3.75,5) -- (5,3.75);
  \draw[-] (4.0,5) -- (5,4.0);
  \draw[-] (4.25,5) -- (5,4.25);
  \draw[-] (4.5,5) -- (5,4.5);
  \draw[-] (4.75,5) -- (5,4.75);
  \draw[-] (5.0,5) -- (5,5.0);
\end{tikzpicture}
}
    \caption{Extension of $\Delta_\lo$ to a copula with $d=2$.}
    \label{fig:copula_2d_lower}
  \end{subfigure}
  \caption{
    Top: the two $d$-dimensional copula diagonals
    \eqref{eq:diagonal_upper} and \eqref{eq:diagonal_lower}
    constructed to prove \eqref{eq:common_upper}
    and \eqref{eq:common_lower}
    respectively in Theorem~\ref{thm:common}.
    For the upper bound
    (\subref{fig:copula_upper}), the increment over
    $(u, u + \delta]$ is maximized, while for the lower bound
    (\subref{fig:copula_lower}) it is minimized.
    Bottom:
    contour plots for possible two-dimensional
    ($d=2$) copulas
    (\subref{fig:copula_2d_upper}) and (\subref{fig:copula_2d_lower})
    whose diagonals are given
    by $\Delta_\up$ and $\Delta_\lo$ respectively.
    We use the extension due to
    \cite[proof of Theorem~1]{fernandez2018constructions},
    though this is not unique in general.
    Recall that every copula satisfies
    $C(0, \ldots, 0) = 0$ and $C(1, \ldots, 1) = 1$.
  }
  \label{fig:copulas}
\end{figure}

\subsection{Comparisons with other well-known copulas}

We provide some comparisons of the anti-concentration properties
established in Theorem~\ref{thm:common} with those
induced by other well-known copulas (see Figure~\ref{fig:copulas_comparison}).
For simplicity, we restrict to the case that
$X_i \sim \cU$ are uniformly distributed;
extensions to arbitrary common laws
proceed using a straightforward quantile transform,
as in the proof of Theorem~\ref{thm:common}.

\begin{example}[Independence copula]
  If $X_i \sim \cU$ are independent
  for $i \in [d]$, then we have that
  $\P_\ind \bigl( x < \max_{i \in [d]} X_i \leq x + \varepsilon \bigr)
  = (x + \varepsilon)^d - x^d$.
  Taking $x \in (0, 1)$ and $\varepsilon \leq \frac{x}{d-1} \land (1-x)$,
  \begin{align*}
    \P_\ind \biggl( x < \max_{i \in [d]} X_i \leq x + \varepsilon \biggr)
    \leq
    d \varepsilon (x + \varepsilon)^{d-1}
    \leq
    d \varepsilon x^{d-1} \Bigl(1 + \frac{1}{d-1}\Bigr)^{d-1}
    \leq
    e d \varepsilon x^{d-1}.
  \end{align*}
  In contrast, the law $\P_\up$
  attaining the maximum in \eqref{eq:common_upper} has
  $\P_\up \bigl(x < \max_{i \in [d]} X_i \leq x + \varepsilon\bigr)
  = d \varepsilon$; its local concentration probability is greater
  by a factor of at least $x^{1-d}/e \to \infty$ as $d \to \infty$.

  If instead one takes $\varepsilon \in (1/d, 1)$ and
  $x = 1 - \varepsilon$, then for
  the independence copula one obtains
  $\P_\ind \bigl( x < \max_{i \in [d]} X_i \leq x + \varepsilon \bigr)
  = 1 - (1 - \varepsilon)^d$
  while \eqref{eq:common_upper} gives
  $\P_\up \bigl( x < \max_{i \in [d]} X_i \leq x + \varepsilon \bigr) = 1$.
  Although both exhibit concentration of the maximum statistic
  at $x=1$ as expected,
  the independence
  copula does not attain exact concentration.

  Regarding lower bounds, if $x \in (0, 1)$ and $\varepsilon \in (0, 1-x]$,
  then the independence copula gives
  $\P_\ind \bigl( x < \max_{i \in [d]} X_i \leq x + \varepsilon \bigr)
  \geq d \varepsilon x^{d-1} > 0$.
  In contrast, whenever $\varepsilon \leq \frac{d-1}{d}(1-x)$,
  the law $\P_\lo$ achieving the minimum in \eqref{eq:common_lower} satisfies
  $\P_\lo \bigl( x < \max_{i \in [d]} X_i \leq x + \varepsilon \bigr) = 0$.
\end{example}

\begin{example}[Fr{\'e}chet--Hoeffding upper bound]
  Write $\P_\FHU$ for the joint law of $X_1 = \cdots = X_d \sim \cU$.
  For $x \in [0, 1)$ and $\varepsilon \in [0, 1-x]$, we have
  $\P_\FHU \bigl( x < \max_{i \in [d]} \leq x + \varepsilon \bigr)
  = \varepsilon$.
  Since $0 \lor \bigl(1 - x - d(1 - x - \varepsilon)\bigr) \leq
  \varepsilon \leq (d \varepsilon) \land (x + \varepsilon)$,
  the Fr{\'e}chet--Hoeffding upper bound copula
  interpolates between the upper bound \eqref{eq:common_upper}
  and the lower bound \eqref{eq:common_lower} of Theorem~\ref{thm:common}.
\end{example}

\begin{example}[Fr{\'e}chet--Hoeffding lower bound]
  Let $\P_\FHL$ be any joint distribution of $X_i \sim \cU$
  for $i \in [d]$ with copula diagonal satisfying
  $\Delta(x) = 0 \lor (d \cdot x - d + 1)$.
  If $x \geq \frac{d-1}{d}$ and $\varepsilon \in [0, 1-x]$, then
  $\P_\FHL \bigl( x < \max_{i \in [d]} X_i \leq x + \varepsilon \bigr)
  = d \varepsilon$,
  matching \eqref{eq:common_upper}. If
  $x \in \bigl[0,  \frac{d-1}{d}\bigr]$ and
  $\varepsilon \in \bigl[0, \frac{d-1}{d} - x\bigr]$, then
  $\P_\FHL \bigl( x < \max_{i \in [d]} X_i \leq x + \varepsilon \bigr) = 0$,
  agreeing with \eqref{eq:common_lower}.
\end{example}

\begin{figure}[H]
  \centering
  \begin{subfigure}{0.32\textwidth}
    \centering
    \scalebox{\figurescale}{%
\begin{tikzpicture}
  \draw (0,0) rectangle (3.5,3.5);
  \draw (0,0) node[below] {$0$};
  \draw (3.5,0) node[below] {$1$};
  \draw (0,3.5) node[left] {$1$};
  \draw (0,0) node[left] {$0$};
  \draw (0,-0.1) node[below] {\phantom{$\frac{d-1}{d}$}};
  \draw plot coordinates {
    (0.0, 0.0) (0.0875, 1.3671875000000002e-06) (0.175,
    2.1875000000000003e-05) (0.2625, 0.0001107421875) (0.35,
    0.00035000000000000005) (0.4375, 0.0008544921875) (0.525,
    0.001771875) (0.6125, 0.0032826171874999993) (0.7,
    0.005600000000000001) (0.7875, 0.0089701171875) (0.875,
    0.013671875) (0.9625, 0.020016992187500006) (1.05, 0.02835)
    (1.1375, 0.039048242187500005) (1.225, 0.05252187499999999)
    (1.3125, 0.0692138671875) (1.4, 0.08960000000000001) (1.4875,
    0.11418886718749997) (1.575, 0.143521875) (1.6625,
    0.17817324218749997) (1.75, 0.21875) (1.8375,
    0.26589199218750004) (1.925, 0.3202718750000001) (2.0125,
    0.3825951171874999) (2.1, 0.4536) (2.1875, 0.5340576171875)
    (2.275, 0.6247718750000001) (2.3625, 0.7265794921875002) (2.45,
    0.8403499999999998) (2.5375, 0.9669857421875) (2.625,
    1.107421875) (2.7125, 1.2626263671875) (2.8, 1.4336000000000002)
    (2.8875, 1.6213763671874997) (2.975, 1.8270218749999996) (3.0625,
    2.0516357421875) (3.15, 2.29635) (3.2375, 2.5623294921875006)
    (3.325, 2.8507718749999995) (3.4125, 3.1629076171874995) (3.5, 3.5)
  };
\end{tikzpicture}
}
    \caption{Independence}
  \end{subfigure}
  \begin{subfigure}{0.32\textwidth}
    \centering
    \scalebox{\figurescale}{%
\begin{tikzpicture}
  \draw (0,0) rectangle (3.5,3.5);
  \draw (0,0) node[below] {$0$};
  \draw (3.5,0) node[below] {$1$};
  \draw (0,3.5) node[left] {$1$};
  \draw (0,0) node[left] {$0$};
  \draw (0,-0.1) node[below] {\phantom{$\frac{d-1}{d}$}};
  \draw (0,0) -- (3.5,3.5);
\end{tikzpicture}
}
    \caption{Upper Fr{\'e}chet--Hoeffding}
  \end{subfigure}
  \begin{subfigure}{0.32\textwidth}
    \centering
    \scalebox{\figurescale}{%
\begin{tikzpicture}
  \draw (0,0) rectangle (3.5,3.5);
  \draw (0,0) node[below] {$0$};
  \draw (3.5,0) node[below] {$1$};
  \draw (0,3.5) node[left] {$1$};
  \draw (0,0) node[left] {$0$};
  \draw[-] (2,0.1) -- (2,-0.1) node[below] {$\frac{d-1}{d}$};
  \draw (2,0) -- (3.5,3.5);
\end{tikzpicture}
}
    \caption{Lower Fr{\'e}chet--Hoeffding}
  \end{subfigure}
  \caption{The diagonal sections of three well-known
  $d$-dimensional copulas.}
  \label{fig:copulas_comparison}
\end{figure}
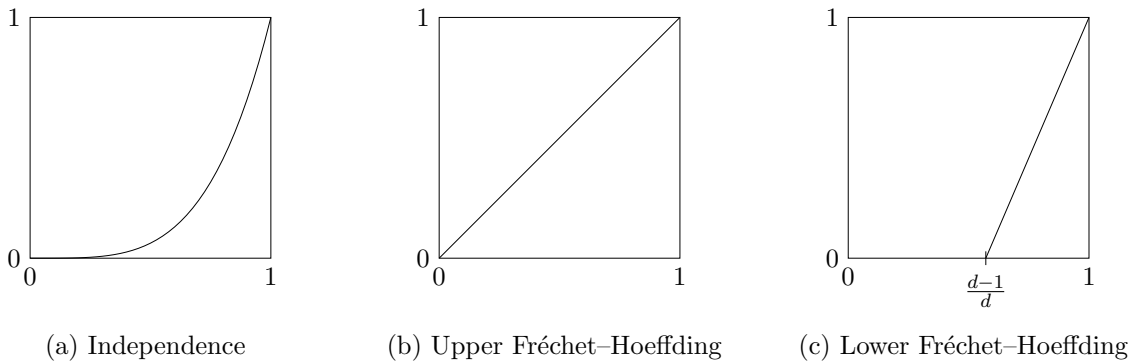

\section{Anti-concentration inequalities for diagonally convex copulas}
\label{sec:convex}

The upper bound presented as \eqref{eq:common_upper} in
Theorem~\ref{thm:common} demonstrates that, without imposing
further conditions on the
dependence structure (the copula) of the random vector $(X_1, \ldots, X_d)$,
it is impossible to obtain anti-concentration results
which hold uniformly over $\varepsilon \geq 0$ and exhibit
a strictly sublinear dependence on the dimension
(see Example~\ref{ex:gaussian_linear}).
As such, in order to obtain sharper upper bounds on
the concentration probability, it is necessary to restrict the
class of admissible copulas.
For example, as discussed in Section~\ref{sec:arbitrary},
in the setting where $(X_1, \ldots, X_d)$
follows a multivariate Gaussian law,
Nazarov's inequality can be applied to the maximum statistic
and produces a bound \eqref{eq:nazarov} with a square root-logarithmic
dependence on the dimension.

Nonetheless, in this section, we propose a method which avoids the
assumption of multivariate (joint) Gaussianity,
replacing it with a mild nonparametric convexity condition on the
copula describing the dependence structure
(see Definition~\ref{def:diagonally_convex}).
We also allow for an arbitrary common marginal distribution;
as such, we encompass a substantially wider range of joint distributions
than those covered by Nazarov's inequality.
See the upcoming
Examples~\ref{ex:convex_gaussian},
\ref{ex:multivariate_gaussian},
\ref{ex:weibull_convex}, \ref{ex:gumbel_convex} and \ref{ex:pareto_convex}
for a selection of novel anti-concentration
inequalities derived using our results.

\begin{definition}
  \label{def:diagonally_convex}
  Let $d \in \N$ and $F: \R \to [0, 1]$ be a CDF.
  Write $\cP_d^\conv(F)$ for the set of distributions
  $\P$ on $\R^d$ that have joint CDFs of the form
  \begin{align*}
    \P \bigl( X_1 \leq x_1, \ldots, X_d \leq x_d \bigr)
    &= C\bigl(F(x_1), \ldots, F(x_d)\bigr),
  \end{align*}
  where $C$ is a $d$-dimensional copula for which
  $\Delta: [0, 1] \to [0, 1]$ defined by
  $\Delta(x) = C(x, \ldots, x)$ is a convex function.
  We say that $\P$ and $C$ are \emph{diagonally convex}.
\end{definition}

\begin{theorem}%
  \label{thm:convex}
  Let $d \in \N$ and $F: \R \to [0, 1]$ be a CDF.
  For each $x \in \R$ and $\varepsilon \geq 0$,
  \begin{align*}
    \max_{\P \in \cP_d^\conv(F)}
    \P \biggl( x < \max_{i \in [d]} X_i \leq x + \varepsilon \biggr)
    &=
    \bigl( F(x + \varepsilon) - F(x) \bigr)
    \biggl\{
      \frac{1}{1 - F(x)}
      \land d
    \biggr\}.
  \end{align*}
\end{theorem}

The upper bound given in Theorem~\ref{thm:convex} holds
uniformly over all diagonally convex copulas
and, moreover, imposes no conditions on the common marginal law $F$.
The proof of Theorem~\ref{thm:convex} is presented in
Section~\ref{sec:proof_thm_convex} and relies only on
convexity arguments.

In cases where $F$ admits a decreasing Lebesgue density $f$,
the dimension-dependence of an anti-concentration bound
derived using Theorem~\ref{thm:convex} is determined by the quantity
$H(x) \vcentcolon= h(x) \land \bigl\{d \cdot f(x)\bigr\}$,
where $h(x) \vcentcolon= f(x) / \bigl(1-F(x)\bigr)$
is the hazard function (or inverse Mills ratio) associated with $F$.
Typically, if $h$ is an increasing function
(sometimes referred to as an ``increasing failure rate'' condition),
then the maximum value of $H(x)$ is attained at a point $x^*$ with
$h(x^*) = d \cdot f(x^*)$, or equivalently with $x^* = F^{-1}(1 - 1/d)$,
yielding a uniform upper bound of
$H(x) \leq d \cdot f\bigl( F^{-1}(1 - 1/d)\bigr)$.
If the hazard function $h$ is instead decreasing
(known as a ``decreasing failure rate'' condition),
then, generally, a dimension-independent bound is obtained;
see Section~\ref{sec:examples_specific} for examples.

\subsection{Examples of diagonally convex copulas}

Before applying Theorem~\ref{thm:convex} with some explicit marginal
laws, we first verify that several popular copula families
satisfy the convex diagonal section condition given in
Definition~\ref{def:diagonally_convex}.
Naturally, the two copulas
\eqref{eq:common_upper} and \eqref{eq:common_lower}
constructed in Theorem~\ref{thm:common}
do \emph{not} generally satisfy this assumption, as evidenced by
the plots of their diagonal sections presented in Figure~\ref{fig:copulas}.
We verify in Example~\ref{ex:frechet-hoeffding-convex}
that diagonal convexity does hold for the independence copula, the
Fr{\'e}chet--Hoeffding upper bound copula,
and any copula with diagonal section matching the
Fr{\'e}chet--Hoeffding lower bound;
see Figure~\ref{fig:copulas_comparison}.

\begin{example}[Diagonally convex copulas]
  \label{ex:frechet-hoeffding-convex}
  The $d$-dimensional independence copula has diagonal section
  $\Delta_\ind(u) = u^d$ and is diagonally convex.
  Similarly, the $d$-dimensional Fr{\'e}chet--Hoeffding upper
  bound copula has diagonal
  $\Delta_\FHU(u) = u$ and is diagonally convex.
  Any copula with diagonal matching the
  $d$-dimensional Fr{\'e}chet--Hoeffding lower bound
  has $\Delta_\FHL(u) = 0 \lor (d u - d + 1)$ and is diagonally convex.
\end{example}

Next, Lemma~\ref{lem:gaussian_copula_convex} demonstrates that
every multivariate Gaussian copula is diagonally convex.
The proof of this result is given in
Section~\ref{sec:proof_lem_gaussian_copula_convex},
and depends on a precise characterization of the Lebesgue density
associated with the maximum statistic of a
multivariate Gaussian distribution \citep[Lemmas~5
and~6]{chernozhukov2015comparison}.

\begin{lemma}
  \label{lem:gaussian_copula_convex}
  Let $\mu \in \R^d$ and suppose $\Sigma \in \R^{d \times d}$
  is a symmetric, positive semi-definite matrix.
  Then $\mathcal N(\mu, \Sigma)$ has a diagonally convex copula.
\end{lemma}

We now give a general condition under which every member of
a family of Archimedean copulas possesses a convex diagonal section.
Let $\psi: [0, 1] \to [0, \infty]$
be a continuous, strictly decreasing function with
$\psi(0) = \infty$ and $\psi(1) = 0$, and suppose that
its inverse function $\psi^{-1}$ is
completely monotonic on $[0, \infty)$, as defined by
\citet[Definition~4.6.1]{nelsen2006}.
Then for each $d \geq 1$, the function
from $[0, 1]^d$ to $[0, 1]$ given by
\begin{align}
  \label{eq:archimedean}
  C(x_1, \ldots, x_d)
  &\vcentcolon=
  \psi^{-1} \Biggl( \sum_{i=1}^{d} \psi(x_i) \Biggr)
\end{align}
is a $d$-dimensional copula, known as the
\emph{Archimedean copula with generator $\psi$}.

\begin{lemma}
  \label{lem:convex_archimedean}
  Let $C$ be a $d$-dimensional Archimedean copula
  with generator $\psi$.
  Suppose that the function $\Psi : (0, \infty) \to \R$, defined by
  \begin{align*}
    \Psi(x)
    &\vcentcolon= \frac{d \cdot \psi' \circ \psi^{-1}(x)}
    {\psi' \circ \psi^{-1}(d \cdot x)}
    = \frac{\bigl(\psi^{-1}\bigr)'(d \cdot x)}
    {\bigl(\psi^{-1}\bigr)'(x)},
  \end{align*}
  is non-increasing. Then $C$ is diagonally convex.
\end{lemma}

We verify in the next example that several popular
families of Archimedean copulas
\citep[Examples~4.23--4.25]{nelsen2006} satisfy the conditions of
Lemma~\ref{lem:convex_archimedean}, and hence are
diagonally convex. The details are contained in
Section~\ref{sec:proofs_examples}.

\begin{example}[Archimedean copulas]
  \label{ex:archimedean}
  The Clayton copulas are Archimedean with generator
  $\psi(x) = x^{-\theta} - 1$ for $\theta > 0$,
  the Frank copulas have generator
  $\psi(x) = \log \frac{e^{-\theta}-1}{e^{-\theta x}-1}$
  for $\theta > 0$, and the Gumbel--Hougaard copulas
  have generator $\psi(x) = (- \log x)^\theta$ for $\theta \geq 1$.
  All of these are diagonally convex.
\end{example}

The condition imposed in Lemma~\ref{lem:convex_archimedean}
is not automatically satisfied. For example, consider
taking $\psi(x) = e^{1/x} - e$ on $(0, 1]$ with
$\psi(0) = \infty$.
Then $\psi^{-1}(x) = 1/\log(x+e)$ which is completely
monotonic on $[0, \infty)$ by
\citet[Equation~14]{miller2001completely}.
Moreover, $(\psi^{-1})'(x) = -1 / \bigl\{(x+e) \log(x + e)^2\bigr\}$
so $\Psi(x) = (x+e) \log(x + e)^2 /
\bigl\{(d \cdot x + e) \log(d \cdot x + e)^2\bigr\}$.
Observe that $\Psi(x) \to 1/d$ as $x \to \infty$ and that,
for $d \geq 2$, we have
$\Psi(2) \leq (2+e) \log(2 + e)^2 /
\bigl\{(2d + e) \log(2d + e)^2\bigr\} \leq 0.95 / d$.
Therefore $\Psi$ cannot be a non-increasing function
when $d \geq 2$.
Further, the resulting copulas fail to possess convex diagonal sections as
$\Delta(x) = 1 / \log\bigl(d e^{1/x} - (d-1) e\bigr)$
is non-convex on $[0, 1]$.

As a final example of a method for constructing diagonally convex
copulas, we consider a model based on mixtures of copulas.
This approach has applications in
dependence-based clustering \citep{arakelian2014clustering}.

\begin{example}[Mixture copula]
  \label{ex:mixture_copula}
  Take $K, d \in \N$ and suppose
  $p_1, \ldots, p_K \geq 0$ are such that $\sum_{k=1}^K p_k = 1$.
  For each $k \in [K]$,
  let $C_k$ be a $d$-dimensional diagonally convex copula.
  Then for $(x_1, \ldots, x_d) \in [0, 1]^d$, the mixture copula
  $(x_1, \ldots, x_d) \mapsto
  \sum_{k=1}^K p_k C_k(x_1, \ldots, x_d)$
  is diagonally convex.
\end{example}

\subsection{Examples with specific marginal distributions}
\label{sec:examples_specific}

Having established the existence of several copulas with
convex diagonal sections, we now demonstrate the application
of Theorem~\ref{thm:convex} with a selection of different
common marginal laws. In Example~\ref{ex:convex_gaussian}
we consider the normal distribution;
see Section~\ref{sec:proofs_examples} for details.

\begin{example}[Marginal Gaussian distribution]
  \label{ex:convex_gaussian}
  Let $d \in \N$, $\mu \in \R$ and $\sigma^2 > 0$.
  Take $(X_1, \ldots, X_d)$ with
  $X_i \sim \mathcal N\bigl(\mu, \sigma^2\bigr)$ for each $i \in [d]$,
  and suppose it is diagonally convex.
  Then for $x \in \R$ and $\varepsilon \geq 0$,
  \begin{align*}
    \P \biggl( x < \max_{i \in [d]} X_i \leq x + \varepsilon \biggr)
    &\leq
    \frac{\varepsilon}{\sigma}
    \Bigl( \sqrt{2 \log d} + 1 \Bigr).
  \end{align*}
\end{example}

We do not require $(X_1, \ldots, X_d)$ to be jointly
Gaussian in Example~\ref{ex:convex_gaussian};
any copula with a convex diagonal section suffices.
In particular, a square root-logarithmic dependence on the dimension
holds regardless of the form of the copula;
for example, any of the copulas described in
Examples~\ref{ex:frechet-hoeffding-convex},
\ref{ex:archimedean} and~\ref{ex:mixture_copula} are permitted.
Therefore, this example offers a version of Nazarov's inequality
\eqref{eq:nazarov} for non-Gaussian joint distributions,
with an improved constant.
Since $\varepsilon$ and $\sigma$ are scale parameters,
\citet[Example~2]{chernozhukov2015comparison} gives an
asymptotic lower bound of $(\varepsilon / \sigma) \sqrt{2 \log d}$
when $X_1, \ldots, X_d$ are known to be independent.

In particular, combining~Example \ref{ex:convex_gaussian} with
Lemma~\ref{lem:gaussian_copula_convex} allows us to deduce the
following result for the maximum of
identically distributed and jointly Gaussian random variables
(and also for their maximum absolute deviation from the mean).

\begin{example}[Joint Gaussian distribution]
  \label{ex:multivariate_gaussian}
  Let $d \in \N$, $\mu \in \R$ and $\sigma^2 > 0$.
  Suppose $(X_1, \ldots, X_d)$ is multivariate Gaussian,
  with $X_i \sim \mathcal N(\mu, \sigma^2)$ for each $i \in [d]$.
  Then for any $x \in \R$ and $\varepsilon \geq 0$,
  \begin{align*}
    \P \biggl( x < \max_{i \in [d]} X_i \leq x + \varepsilon \biggr)
    &\leq
    \frac{\varepsilon}{\sigma}
    \Bigl( \sqrt{2 \log d} + 1 \Bigr), \\
    \P \biggl( x < \max_{i \in [d]} |X_i - \mu| \leq x + \varepsilon \biggr)
    &\leq
    \frac{\varepsilon}{\sigma}
    \Bigl( \sqrt{2 \log 2d} + 1 \Bigr).
  \end{align*}
\end{example}

In Example~\ref{ex:weibull_convex} we apply
Theorem~\ref{thm:convex} to a family of Weibull distributions.
\begin{example}[Weibull distribution]
  \label{ex:weibull_convex}
  Let $d \in \N$, $\alpha \geq 1$ and $\lambda > 0$.
  Suppose $(X_1, \ldots, X_d)$ is a random vector
  with a diagonally convex copula, and
  $\P(X_i \leq x) = 1 - \exp\bigl(-(x/\lambda)^\alpha\bigr)$
  for $x \geq 0$ and $i \in [d]$.
  Then for each $x \geq 0$ and $\varepsilon \geq 0$,
  \begin{align*}
    \P \biggl( x < \max_{i \in [d]} X_i \leq x + \varepsilon \biggr)
    &\leq
    \frac{\varepsilon \alpha}{\lambda}
    \bigl(\log d + 1\bigr)^{\frac{\alpha-1}{\alpha}}.
  \end{align*}
\end{example}
The dimension dependence in Example~\ref{ex:weibull_convex}
is poly-logarithmic, with
the exponent depending on the value of the shape parameter
$\alpha$.
If $X_1, \ldots, X_d$ are independent with
common distribution function $F(x)$ and density $f(x)$ then
the density of $\max_{i \in [d]} X_i$ is $d F(x)^{d-1} f(x)$;
applying this with
$x = (\log d)^{1/\alpha}$ gives a lower bound of
$\frac{\varepsilon \alpha}{e \lambda} (\log d)^{\frac{\alpha-1}{\alpha}}$.
With $\alpha = 1$, we recover the exponential distribution,
and the bound given in Example~\ref{ex:weibull_convex}
reduces to $\varepsilon / \lambda$.
This dimension-independent result arises because
the hazard function is constant.
For $\alpha > 1$, the hazard function is increasing,
yielding a dimension-dependent bound.
When $\alpha = 2$, we recover a Rayleigh distribution and the
dimension dependence scales as $\sqrt{\log d}$;
the same as for the Gaussian distribution
(Example~\ref{ex:convex_gaussian}).

Next, we consider a family of reverse Gumbel distributions.

\begin{example}[Reverse Gumbel distribution]
  \label{ex:gumbel_convex}
  Let $d \in \N$ and $\lambda > 0$.
  Suppose that $(X_1, \ldots, X_d)$ is a random vector
  with a diagonally convex copula and
  $\P(X_i \leq x) = 1 - \exp\bigl(-e^{x / \lambda}\bigr)$
  for $x \in \R$ and $i \in [d]$.
  Then for any $x \in \R$ and $\varepsilon \geq 0$,
  \begin{align*}
    \P \biggl( x < \max_{i \in [d]} X_i \leq x + \varepsilon \biggr)
    &\leq
    \frac{\varepsilon}{\lambda}
    \bigl(1 + \log d\bigr).
  \end{align*}
\end{example}

For reverse Gumbel distributions, the hazard function is
increasing, giving a dimension-dependent bound,
here on the order of $\log d$.
By considering the density of $\max_{i \in [d]} X_i$
at $x = 0 \lor \log \log d$ with $X_1, \ldots, X_d$ independent,
a lower bound of $\frac{\varepsilon}{e \lambda} \log d$ can be derived.

In the next example, we consider a family of Pareto distributions.
\begin{example}[Pareto distribution]
  \label{ex:pareto_convex}
  Let $d \in \N$, $\alpha > 0$ and $\lambda > 0$.
  Suppose $(X_1, \ldots, X_d)$ is a random vector
  with a diagonally convex copula and
  $\P(X_i \leq x) = 1 - (\lambda/x)^\alpha$
  for $x \geq \lambda$.
  For $x \geq 0$ and $\varepsilon \geq 0$,
  \begin{align*}
    \P \biggl( x < \max_{i \in [d]} X_i \leq x + \varepsilon \biggr)
    &\leq
    \frac{\alpha \varepsilon}{\lambda}.
  \end{align*}
\end{example}
Since Pareto distributions have decreasing hazard functions,
the resulting bound in Example~\ref{ex:pareto_convex} is
dimension-independent.
A lower bound can be obtained by considering $X_1 = \cdots = X_d$.
Since the Pareto density at $x = \lambda$ is $\alpha / \lambda$,
we derive a lower bound of $\alpha \varepsilon / \lambda$.
Alternatively, if $X_1, \ldots, X_d$ are independent
Pareto random variables, then
the concentration function admits an upper bound
of the order $d^{-1/\alpha} \varepsilon / \lambda$,
up to a constant depending on $\alpha$
\citep[Example~3.3.10]{embrechts2013modelling}.

Our final example concerns a family of gamma distributions.
\begin{example}[Gamma distribution]
  \label{ex:gamma}
  Let $d \in \N$ and take $\alpha \geq 1$ and $\lambda > 0$.
  Suppose that $(X_1, \ldots, X_d)$
  is a random vector with a diagonally convex copula
  such that, for each $j \in [d]$, the Lebesgue density of $X_j$ is
  $f(x) = x^{\alpha - 1} e^{-x / \lambda} \lambda^{-\alpha} / \Gamma(\alpha)$
  on $x \geq 0$.
  Then for $x \geq 0$ and $\varepsilon \geq 0$,
  \begin{align*}
    \P \biggl( x < \max_{j \in [d]} X_j \leq x + \varepsilon \biggr)
    &\leq
    \frac{\varepsilon}{\lambda}.
  \end{align*}
\end{example}
Gamma distributions have bounded hazard functions,
yielding a dimension-independent bound in Example~\ref{ex:gamma}.
Since the gamma density at $x = \lambda (\alpha - 1)$ is
$\lambda^{- 1} (\alpha - 1)^{\alpha - 1} e^{1 - \alpha}
/ \Gamma(\alpha) \geq (2 \pi \alpha)^{-1/2} / \lambda$,
we have a lower bound of
$(2 \pi \alpha)^{-1/2}\varepsilon / \lambda$
for the concentration function when $X_1 = \cdots = X_d$.
If instead $\alpha \in (0, 1)$,
the density of $X_1$ is unbounded on every neighborhood of zero,
rendering uniform linear-in-$\varepsilon$
anti-concentration inequalities impossible.

\section{Application to high-dimensional statistical inference}
\label{sec:application}

We illustrate the applicability of our results with an example of a statistical
inference procedure using a potentially high-dimensional test statistic.
Let $X$ be an $\R^d$-valued random vector
constructed using samples taken from an underlying data set.
For example, $X$ might represent (an appropriate transformation of)
the fitted coefficients of a parametric model
or a discretized version of a nonparametric estimator.
Since weak convergence of the law of $X$ routinely
fails in high-dimensional settings,
we suppose instead that a
coupling (strong approximation) for $X$ is available
\citep[see, for example,][and references therein]{%
  chernozhukov2013gaussian,%
  chernozhukov2014anti,%
  chernozhukov2014gaussian,%
  cattaneo2022yurinskii,%
  cattaneo2024uniform,%
  cattaneo2024strong%
}.
That is, there exists an $\R^d$-valued random
vector $T=(T_1,\ldots,T_d)$, on the same probability space as
$X=(X_1,\ldots,X_d)$, with
\begin{align*}
  \P \bigl( \| X - T \|_\infty > \varepsilon \bigr)
  \leq p(\varepsilon)
\end{align*}
for some decreasing function $p: [0, \infty) \to [0, 1]$,
where $\|x\|_\infty \vcentcolon= \max_{i \in [d]} |x_i|$.
Typically, either one knows the law of $T$
explicitly, or can draw samples from it.
Inference proceeds by choosing a significance level
$\alpha \in (0, 1)$ and computing a quantile
$q_\alpha \vcentcolon= \inf\bigl\{ q\in\mathbb{R} :
\P \bigl(\max_{i \in [d]} T_i \leq q \bigr) \geq 1-\alpha\bigr\}$.
It is straightforward to verify that for all $\varepsilon \geq 0$,
\begin{align*}
  \Bigl|
  \P \Bigl(
    \max_{i \in [d]} X_i
    > q_\alpha
  \Bigr)
  - \alpha
  \Bigr|
  \leq
  p(\varepsilon)
  + \Bigl\{
    \P \Bigl(
      q_\alpha - \varepsilon < \max_{i \in [d]} T_i \leq q_\alpha
    \Bigr)
    \lor \P \Bigl(
      q_\alpha < \max_{i \in [d]} T_i \leq q_\alpha + \varepsilon
    \Bigr)
  \Bigr\};
\end{align*}
this bound can then be minimized over $\varepsilon \geq 0$.
In this setting,
it suffices to control the anti-concentration terms
in a neighborhood of the quantile $q_\alpha$;
our Theorem~\ref{thm:convex} therefore provides sharper bounds
than those derived using the L{\'e}vy concentration function
$L\bigl(\max_{i \in [d]} T_i, \varepsilon\bigr)$.
As such, the validity of the resulting test based on $X$
relies on the availability of both
\begin{inlineroman}
  \item a tight coupling inequality for $\|X - T\|_\infty$, and
  \item a sharp anti-concentration bound for $\max_{i \in [d]} T_i$.
\end{inlineroman}

Our main anti-concentration results given in Section~\ref{sec:arbitrary}
and Section~\ref{sec:convex} can be applied whenever the entries
$T_1, \ldots, T_d$ are identically distributed.
While this is, in principle, a restrictive assumption,
it is often satisfied in practice for the purpose of maximizing
statistical power against alternative hypotheses.
For example, consider the canonical setting in which
the law of $X$ is approximated using a multivariate Gaussian
random vector $T$. It is usual to standardize the entries
of $X$ (and therefore also the corresponding entries of $T$)
so that $\E[X_i] = 0$ and $\E\bigl[X_i^2\bigr] = 1$;
it follows immediately that each $T_i$ is distributed
marginally as $\cN(0, 1)$.

Our results are applicable to a substantially broader class of
inference procedures than existing approaches, which typically
require joint Gaussianity of the entries of $T$.
Specifically, under Theorem~\ref{thm:convex}
we allow for an arbitrary common marginal distribution
as well as a wide range of joint distributions, so long as the
underlying copula is
diagonally convex (Definition~\ref{def:diagonally_convex}).
Moreover, the resulting anti-concentration inequality
is agnostic to the precise form of the copula, allowing for
settings where the joint distribution of $T$ is unknown or difficult
to estimate.

\subsection{Gaussian mixture approximations for factor models}

For a specific example in a non-Gaussian regime, we
consider the Gaussian mixture approximation for
martingale factor models discussed by
\cite[Section~2.5]{cattaneo2022yurinskii}.
In this setting, a size-$n$ sample of
$d$-dimensional observations $\bigl(X^{(1)}, \ldots, X^{(n)}\bigr)$
is taken from an underlying random process and is assumed to
form a zero-mean martingale. For example,
$X^{(1)}, \ldots, X^{(n)}$ may denote multivariate quantities
derived from a time series estimation procedure. We impose further
structure by means of a factor model;
take $m \in \N$ and suppose that
\begin{align*}
  X^{(i)} = L g^{(i)} + \varepsilon^{(i)}
\end{align*}
for $i \in [n]$,
where $L$ takes values in $\R^{d \times m}$, $g^{(i)}$ in $\R^m$,
and $\varepsilon^{(i)}$ in $\R^d$.
We interpret $g^{(i)}$ as a latent factor variable
and $L$ as a random factor loading, with
independent disturbances $\varepsilon^{(i)}$.
We assume that
$\varepsilon^{(i)}$ is zero-mean and finite-variance for each $i \in [n]$, and
that $\bigl(\varepsilon^{(1)}, \ldots, \varepsilon^{(n)}\bigr)$ is independent
of $L$ and $\bigl(g^{(1)}, \ldots, g^{(n)}\bigr)$.
Suppose that
$\E \bigl[ g^{(i)} \mid L, g^{(1)}, \ldots, g^{(i-1)} \bigr] = 0$
for each $i \in [n]$.

Corollary~2.2 in \cite{cattaneo2022yurinskii} provides
sufficient conditions for a coupling
between $\sum_{i=1}^{n} X^{(i)}$ and an $\R^d$-valued random vector $T$
with conditional distribution $T \mid L \sim \cN(0, \Sigma)$,
where
\begin{align*}
  \Sigma \vcentcolon=
  \sum_{i=1}^n \Bigl(L \Var\bigl[g^{(i)} \mid L\bigr] L^\T
  + \Var\bigl[\varepsilon^{(i)}\bigr]\Bigr).
\end{align*}

We now impose some further conditions on the law of $T$;
in particular, we ensure that the copula associated with $T$ is
the independence copula and that the
entries $T_j$ share a common marginal distribution.
Therefore, suppose that the Lebesgue density function
$f: \R^d \to [0, \infty)$ of $T$ satisfies
\begin{align}
  \label{eq:gmm}
  f(x_1, \ldots, x_d)
  &=
  \prod_{j=1}^d
  \Biggl(
    \sum_{k=1}^K \frac{p_k}{\sigma_k}
    \phi\biggl(\frac{x_j}{\sigma_{k}}\biggr)
  \Biggr)
  = \sum_{k_1 = 1}^K
  \cdots
  \sum_{k_d = 1}^K
  \Biggl(
    \prod_{j=1}^d
    p_{k_j}
  \Biggr)
  \Biggl(
    \prod_{j=1}^d
    \frac{1}{\sigma_{k_j}}
    \phi\biggl(\frac{x_j}{\sigma_{k_j}}\biggr)
  \Biggr),
\end{align}
where $\sigma_{k} > 0$ and $p_{k} \in [0, 1]$ for $k \in [K]$
satisfy $\sum_{k=1}^K p_{k} = 1$.
The variables $T_1, \ldots, T_d$ are seen to be independent
and identically distributed by the
factorization given in the first equality in~\eqref{eq:gmm}.
Further, $T$ follows a multivariate Gaussian mixture distribution with
$K^d$ components by the second equality in~\eqref{eq:gmm};
in general, $T$ is not a Gaussian random vector.
The copula associated with $T$ is, therefore, the independence
copula, which is diagonally convex by
Example~\ref{ex:frechet-hoeffding-convex}.
In Example~\ref{ex:gmm}, we present an anti-concentration bound
based on Theorem~\ref{thm:convex}
for the maximum statistic of such a distribution,
using the result from Example~\ref{ex:convex_gaussian} and the structure of
the multivariate mixture model.

\begin{example}[Gaussian mixture distribution]
  \label{ex:gmm}
  Let $d, K \in \N$ and suppose $(T_1, \ldots, T_d)$ has
  Lebesgue density as given in \eqref{eq:gmm},
  where $p_{k} \in (0, 1]$ with $\sum_{k=1}^K p_{k} = 1$
  and $0 < \sigma_{k} \leq 1$, for $k \in [K]$.
  We assume that $\sigma_1 = 1$,
  and write $\sigma \vcentcolon= \min_{k \in [K]} \sigma_k$.
  Then for any $x \in \R$ and $\varepsilon \geq 0$,
  \begin{align}
    \label{eq:gmm_bound}
    \P \biggl( x < \max_{j \in [d]} T_j \leq x + \varepsilon \biggr)
    &\leq
    \Biggl\{
      \frac{\varepsilon}{p_1}
      \Biggl(
        \! \sqrt{2 \log d} +
        2 \sum_{k=1}^K
        \frac{p_k}{\sigma_k}
      \Biggr)
    \Biggr\}
    \land
    \biggl\{
      \frac{\varepsilon}{\sigma}
      \Bigl( \sqrt{2 \log d} + 2 \Bigr)
    \biggr\}.
  \end{align}
\end{example}

By applying Nazarov's inequality \eqref{eq:nazarov}
instead of our Theorem~\ref{thm:convex},
one can show that \eqref{eq:gmm_bound} admits an upper bound of
$\frac{\varepsilon}{\sigma}
\bigl( \sqrt{2 \log d} + 2 \bigr)$.
As such, our Example~\ref{ex:gmm} offers a sharper
anti-concentration inequality when the minimum variance
$\sigma$ is small.
For instance, consider the setting where
$\sigma_k = 1$ for $k \in [K-1]$ and $\sigma_K = \sigma \leq 1$,
and suppose for simplicity that $p_k = 1/K$ for all $k \in [K]$.
Then Example~\ref{ex:gmm} gives
\begin{align*}
  \P \biggl( x < \max_{j \in [d]} T_j \leq x + \varepsilon \biggr)
  &\leq
  \varepsilon
  \biggl(
    K \sqrt{2 \log d}
    + 2(K-1)
    + \frac{2}{\sigma}
  \biggr).
\end{align*}
Crucially, the dependence on the dimension $d$
(which may be large in high-dimensional regimes) and the minimum
variance $\sigma^2$
(which could be small in the presence of approximately
degenerate Gaussian mixture components) is additive;
contrast this with the corresponding bound from Nazarov's inequality
\eqref{eq:nazarov}, which contains the multiplicative term
$\frac{\varepsilon}{\sigma} \sqrt{2 \log d}$.

Such a separation between the high-dimensional
and the degenerate variance terms is to be expected in the
resulting anti-concentration inequality for the following reason.
Either
\begin{inlineroman}
  \item $d$ is large: then since $T_i$ are i.i.d.,\ the
    maximum statistic will be realized far from the origin
    with high probability; here, the marginal density
    is dominated by the component with the largest (unit) variance; or
  \item $d$ is small; then the maximum statistic may realize near zero,
    where the component with the smallest variance dominates the density.
\end{inlineroman}
This approach is similar to that taken by
\cite{lopes2020bootstrapping} in the context of bootstrap
approximations under variance decay.

As an alternative to applying our general Theorem~\ref{thm:convex},
one could instead apply Nazarov's inequality after conditioning
on the latent factors,
observing that the minimum variance satisfies
$\sigma^2 \geq \min_{j \in [d]} \sum_{i=1}^n
\Var\bigl[\varepsilon_j^{(i)}\bigr]$
almost surely, by the definition of $\Sigma$.
However, this leads to an anti-concentration inequality
which is worse than that in Example~\ref{ex:gmm}
whenever $\min_{j \in [d]} \Sigma_{j j}$ is
dominated by the factor variance component
$\sum_{i=1}^n \bigl(L \Var[g^{(i)} \mid L] L^\T\bigr)_{j j}$.

\subsection{High-dimensional chi-squared approximations}

As a second non-Gaussian example, we consider a setting
where the $d$-dimensional statistic $X$ is approximated
using chi-squared distributions. To be precise, we suppose the
existence of a coupling between $X$ and a vector $T = (T_1, \ldots, T_d)$
where $T_j \sim \chi^2_p$ for each $j \in [d]$ with $p \in \N$.
Such couplings may be established, for instance, via the application of
tools from the strong approximation literature
\citep{csorgo1981strong,Lindvall_1992_Book,pollard2002user}.
Approximate chi-squared behavior of vector-valued statistics arises in
various contexts, including goodness-of-fit hypothesis
testing based on Pearson’s statistic \citep{gaunt2017stein}, and in the
analysis of degenerate U-statistics \citep{dobler2018gamma}.

To bound the anti-concentration function of the maximum
statistic $\max_{j \in [d]} T_j$,
we do not require the entries to be independent;
rather we assume that the joint law of $(T_1, \ldots, T_d)$ possesses
a diagonally convex copula (Definition~\ref{def:diagonally_convex}).
Example~\ref{ex:gamma} then covers the family of chi-squared distributions
$\chi^2_p$ where $p \in \N$ with $p \geq 2$
as a special case. In particular, if $T_j \sim \chi^2_p$
then $T_j$ follows a gamma distribution with
$\alpha = p/2 \geq 1$ and $\lambda = 2$.
Therefore, for $x \geq 0$ and $\varepsilon \geq 0$,
\begin{align*}
  \P \biggl( x < \max_{j \in [d]} T_j \leq x + \varepsilon \biggr)
  &\leq
  \frac{\varepsilon}{\lambda}.
\end{align*}
Boundedness of the chi-squared hazard function thus yields
dimension-independent anti-concentration inequalities,
under a diagonally convex copula assumption.
\section{Proofs}
\label{sec:proofs}

We give full proofs for all of our results,
along with detailed calculations associated with
examples given in the main paper.

\subsection{Theorem~\ref{thm:common}}
\label{sec:proof_thm_common}

Before proving Theorem~\ref{thm:common}, we first establish
the result in the special case
where the common law is the standard uniform distribution,
in Lemma~\ref{lem:uniform}.
For clarity, we write $U_i \sim \cU$ for each $i \in [d]$.

\begin{lemma}
  \label{lem:uniform}
  Let $d \in \N$. For $u \in [0, 1]$ and $\delta \in [0, 1-u]$,
  \begin{align}
    \label{eq:uniform_upper}
    \max_{\P \in \cP_d(\cU)}
    \P \biggl( u < \max_{i \in [d]} U_i \leq u + \delta \biggr)
    &=
    (d \delta)
    \land (u + \delta), \\
    \label{eq:uniform_lower}
    \min_{\P \in \cP_d(\cU)}
    \P \biggl( u < \max_{i \in [d]} U_i \leq u + \delta \biggr)
    &=
    0 \lor
    \bigl\{1 - u - d (1 - u - \delta)\bigr\}.
  \end{align}
\end{lemma}

\begin{proof}[Proof of Lemma~\ref{lem:uniform}]
  By a union bound, for any $\P \in \cP_d(\cU)$,
  \begin{align*}
    \P \Bigl(u < \max_{i \in [d]} U_i \leq u + \delta\Bigr)
    &\leq
    \P \biggl(
      \{U_1 \leq u + \delta\}
      \cap
      \bigcup_{i \in [d]} \{u < U_i \leq u + \delta\}
    \biggr)
    \leq (d \delta) \land (u + \delta), \\
    \P \Bigl(
      u < \max_{i \in [d]} U_i \leq u + \delta
    \Bigr)
    &\geq
    0 \lor \biggl\{
      1 - \P \biggl(
        \max_{i \in [d]} U_i \leq u
      \biggr)
      - \P\biggl(
        \max_{i \in [d]} U_i > u + \delta
      \biggr)
    \biggr\} \\
    &\geq
    0 \lor
    \bigl\{1 - u - d (1 - u - \delta)\bigr\}.
  \end{align*}

  It remains to show these bounds can be attained.
  Let $\Delta: [0, 1] \to [0, 1]$ be a $d$-dimensional copula diagonal,
  and $C: [0, 1]^d \to [0, 1]$ be a copula satisfying
  $\Delta(t) = C(t, \ldots, t)$ for all $t \in [0, 1]$
  (see Lemma~\ref{lem:copula_diagonal}).
  By Sklar's theorem \citep[Theorem~2.10.9]{nelsen2006},
  $C$ is the joint distribution function of a random vector
  $(U_1, \ldots, U_d)$ where $U_i \sim \cU$ for each $i \in [d]$.
  Further, if $(U_1, \ldots, U_d) \sim C$ then
  \begin{align*}
    \P \biggl( u < \max_{i \in [d]} U_i \leq u + \delta \biggr)
    &=
    C(u + \delta, \ldots, u + \delta) - C(u, \ldots, u)
    = \Delta(u + \delta) - \Delta(u).
  \end{align*}

  For \eqref{eq:uniform_upper}, consider taking $\Delta = \Delta_\up$
  as defined by \eqref{eq:diagonal_upper}.
  Clearly $\Delta_\up(1) = 1$.
  If $t \leq \frac{d u}{d-1}$ then $d(t-u) - t \leq 0$,
  so $\Delta_\up(t) \leq t$ for all $t \in [0, 1]$.
  Further, $\Delta_\up$ is piecewise linear with the gradient of each
  piece bounded below by zero and above by $d$,
  so $0 \leq \Delta_\up(t') - \Delta_\up(t) \leq d(t'-t)$
  whenever $0 \leq t \leq t' \leq 1$.
  Thus $\Delta_\up$ satisfies the conditions of
  Lemma~\ref{lem:copula_diagonal}
  and is a $d$-dimensional copula diagonal.

  If $u \leq \frac{d-1}{d}$, then $\Delta_\up(u) = 0$,
  while $u > \frac{d-1}{d}$ implies $\Delta_\up(u) = d u - d + 1$.
  If $u + \delta \leq \frac{d u}{d-1}$
  (which occurs if and only if $d \delta \leq u + \delta$),
  then $u \leq \frac{d-1}{d}$ gives
  $\Delta_\up(u + \delta) = d \delta$,
  while $u > \frac{d-1}{d}$ implies
  $\Delta_\up(u + \delta) = d u + d \delta - d + 1$.
  Conversely if $u + \delta > \frac{d u}{d-1}$,
  then $u \leq \frac{d-1}{d}$ and
  $\Delta_\up(u + \delta) = u + \delta$.

  For \eqref{eq:uniform_lower}, consider taking $\Delta = \Delta_\lo$
  as defined by \eqref{eq:diagonal_lower}.
  Clearly we have
  $\Delta_\lo(1) = 1$, and $1 - d(1-t) - t \leq 0$, so
  $\Delta_\lo(t) \leq t$ for all $t \in [0, 1]$.
  Further, $\Delta_\lo$ is piecewise linear with each gradient
  bounded below by zero and above by $d$,
  so $0 \leq \Delta_\lo(t') - \Delta_\lo(t) \leq d(t'-t)$
  when $0 \leq t \leq t' \leq 1$.
  Thus $\Delta_\lo$ satisfies the conditions of
  Lemma~\ref{lem:copula_diagonal}
  and is a $d$-dimensional copula diagonal.

  Now observe that $\Delta_\lo(u) = u$.
  If $u + \delta \leq \frac{d+u-1}{d}$
  (if and only if $1 - u - d(1-u-\delta) \leq 0$),
  then $\Delta_\lo(u + \delta) = u$.
  Conversely if $u + \delta > \frac{d+u-1}{d}$,
  then $\Delta_\lo(u + \delta) = 1 - d(1-u-\delta)$.
\end{proof}

\begin{proof}[Proof of Theorem~\ref{thm:common}]
  Let $u = F(x) \in [0, 1]$
  and $\delta = F(x + \varepsilon) - F(x) \in [0, 1-u]$.
  Let $F^{-1}: [0, 1] \to \R$ be the quantile function
  satisfying $F^{-1}(t) \leq s$ if and only if
  $t \leq F(s)$ for all $s \in \R$ and $t \in [0, 1]$.
  If $U \sim \cU$ and $X = F^{-1}(U)$, then
  $\P(X \leq x) = \P\bigl(F^{-1}(U) \leq x\bigr)
  = \P\bigl(U \leq F(x)\bigr) = F(x)$,
  so $X \sim F$. To prove \eqref{eq:common_upper},
  take $(U_1, \ldots, U_d)$ as in
  \eqref{eq:uniform_upper} of Lemma~\ref{lem:uniform}
  and set $X_i = F^{-1}(U_i)$ for $i \in [d]$ so
  \begin{align*}
    \P \biggl( x < \max_{i \in [d]} X_i \leq x + \varepsilon \biggr)
    &= \P \biggl( x < \max_{i \in [d]} F^{-1}(U_i) \leq x + \varepsilon \biggr)
    = \P \biggl( u < \max_{i \in [d]} U_i \leq u + \delta \biggr) \\
    &=
    (d \delta) \land (u + \delta)
    =
    \bigl\{d \bigl(F(x + \varepsilon) - F(x)\bigr)\bigr\}
    \land F(x + \varepsilon).
  \end{align*}
  To show \eqref{eq:common_lower},
  take $(U_1, \ldots, U_d)$ as in
  \eqref{eq:uniform_lower} of Lemma~\ref{lem:uniform}
  and $X_i = F^{-1}(U_i)$ for $i \in [d]$
  so that
  \begin{align*}
    \P \biggl( x < \max_{i \in [d]} X_i \leq x + \varepsilon \biggr)
    &=
    \P \biggl( u < \max_{i \in [d]} U_i \leq u + \delta \biggr)
    =
    0 \lor \bigl(1 - u - d (1 - u - \delta)\bigr) \\
    &=
    0 \lor \bigl\{1 - F(x) - d \bigl(1 - F(x + \varepsilon)\bigr)\bigr\}.
    \qedhere
  \end{align*}
\end{proof}

\subsection{Theorem~\ref{thm:convex}}
\label{sec:proof_thm_convex}

\begin{proof}[Proof of Theorem~\ref{thm:convex}]
  Let $\Delta$ be the copula diagonal section
  associated with the law of $(X_1, \ldots, X_d)$,
  and suppose it is a convex function on $[0, 1]$.
  Since $\Delta(1) = 1$,
  for any $u \in [0, 1]$ and $\delta \in [0, 1-u]$,
  \begin{align*}
    \Delta(u + \delta)
    &= \Delta\biggl(
      \frac{1-u-\delta}{1-u} \cdot u
      + \frac{\delta}{1-u} \cdot 1
    \biggr)
    \leq
    \frac{1-u-\delta}{1-u} \Delta(u)
    + \frac{\delta}{1-u}.
  \end{align*}
  Combining this with the facts that
  $\Delta(u + \delta) - \Delta(u) \leq d \delta$
  and $\Delta(u) \geq 0$ (see Lemma~\ref{lem:copula_diagonal}) gives
  \begin{align*}
    \Delta(u + \delta) - \Delta(u)
    &\leq
    \biggl\{
      \frac{\delta}{1 - u}
      \bigl(1 - \Delta(u)\bigr)
    \biggr\}
    \land (d \delta)
    \leq
    \delta\,
    \biggl(
      \frac{1}{1 - u}
      \land d
    \biggr).
  \end{align*}
  A quantile transform (see the proof of Theorem~\ref{thm:common})
  with $u = F(x)$ and
  $\delta = F(x + \varepsilon) - F(x)$ yields
  \begin{align*}
    \P \biggl( x < \max_{i \in [d]} X_i \leq x + \varepsilon \biggr)
    &=
    \P \biggl( F(x) < \max_{i \in [d]} F(X_i) \leq F(x + \varepsilon) \biggr)
    =
    \Delta \circ F(x + \varepsilon) - \Delta \circ F(x) \\
    &\leq
    \bigl( F(x + \varepsilon) - F(x) \bigr)
    \biggl\{
      \frac{1}{1 - F(x)}
      \land d
    \biggr\}.
  \end{align*}
  We now show that the maximum is attained.
  For $d \in \N$ and $u \in [0, 1]$, define
  \begin{align*}
    \Delta_u(t)
    &\vcentcolon=
    \biggl\{ t - \biggl(u \land \frac{d-1}{d}\biggr) \biggr\}
    \biggl(
      \frac{1}{1-u}
      \land d
    \biggr)
    \cdot \I \biggl\{ u \land \frac{d-1}{d} < t \biggr\},
  \end{align*}
  which is a valid $d$-dimensional
  copula diagonal by Lemma~\ref{lem:copula_diagonal}.
  Note that for $\delta \in [0, 1-u]$, we have
  $\Delta_u(u + \delta) - \Delta_u(u)
  = \delta \bigl(\frac{1}{1-u} \land d\bigr)$.
  Apply the quantile transform as above to conclude.
\end{proof}

\subsection{Lemma~\ref{lem:gaussian_copula_convex}}
\label{sec:proof_lem_gaussian_copula_convex}

\begin{proof}[Proof of Lemma~\ref{lem:gaussian_copula_convex}]
  If $d \geq 2$ and $\Sigma_{11} = 0$,
  by Sklar's theorem \citep[Theorem~2.10.9]{nelsen2006},
  given a copula $C': [0, 1]^{d-1} \to [0, 1]$ for $(X_2, \ldots, X_d)$,
  construct a copula for $(X_1, \ldots, X_d)$ as
  $C(x_1, \ldots, x_d) = x_1 \land C'(x_2, \ldots, x_d)$.
  The diagonal of $C$ is $C(x, \ldots, x) = C'(x, \ldots, x)$.
  If $d = 1$ and $\Sigma_{11} = 0$ then take $C(x) = x$.
  We thus assume without loss of generality
  that $\Sigma_{i i} > 0$ for all $i \in [d]$.

  The copula of $(X_1, \ldots, X_d)$ is the same as that for
  $(Z_1, \ldots, Z_d)$, where $Z_i = (X_i - \mu_i) / \sigma_i$ and
  $\sigma_i^2 = \Sigma_{i i}$ for each $i \in [d]$.
  Therefore, without loss of generality, set
  $\mu = 0$ and $\Sigma_{i i} = 1$ for each $i \in [d]$.

  If $\Sigma_{i j} = 1$ where $1 \leq i < j \leq d$
  then $\max_{i \in [d]} X_i = \max_{i \in [d]\setminus\{j\}} X_i$
  almost surely so, without loss of generality, we take
  $\Sigma_{i j} < 1$ whenever $i \neq j$.
  Lemmas~5 and~6 in \cite{chernozhukov2015comparison} show
  $\max_{i \in [d]} X_i$ has Lebesgue density $f(x) = \phi(x) g(x)$
  where $g: \R \to \R$ is increasing. For $x \in (0, 1)$, note
  $\Delta(x) = \P\bigl(\max_{i \in [d]} X_i \leq \Phi^{-1}(x)\bigr)$,
  so $\Delta$ is differentiable on $(0, 1)$ with
  increasing derivative
  \begin{align*}
    \Delta'(x)
    &=
    \frac{f \circ \Phi^{-1}(x)}{\phi \circ \Phi^{-1}(x)}
    =
    g \circ \Phi^{-1}(x).
  \end{align*}
  Thus $\Delta$ is convex on $(0, 1)$, and also on $[0, 1]$
  as $0 = \Delta(0) \leq \Delta(x) \leq \Delta(1) = 1$
  for all $x \in (0, 1)$.
\end{proof}

\subsection{Lemma~\ref{lem:convex_archimedean}}

\begin{proof}[Proof of Lemma~\ref{lem:convex_archimedean}]
  The diagonal section of the Archimedean copula $C$ is
  \begin{align*}
    \Delta(x)
    &= \psi^{-1}\bigl(d \cdot \psi(x)\bigr).
  \end{align*}
  As $\psi$ is differentiable with non-zero derivative on $(0, 1)$,
  the inverse function theorem gives
  \begin{align*}
    \Delta'(x)
    &= \frac{d \cdot \psi'(x)}
    {\psi' \circ \psi^{-1}\bigl(d \cdot \psi(x)\bigr)}
    = d \cdot \Psi \circ \psi(x).
  \end{align*}
  Since $\psi$ is strictly decreasing and $\Psi$ is non-increasing,
  $\Delta'$ is non-decreasing on $(0, 1)$.
  As $0 = \Delta(0) \leq \Delta(x) \leq \Delta(1) = 1$
  for all $x \in (0, 1)$, we conclude that $\Delta$ is
  convex on $[0, 1]$.
\end{proof}

\subsection{Examples}
\label{sec:proofs_examples}

\begin{proof}[Details for Example~\ref{ex:archimedean}]
  We apply Lemma~\ref{lem:convex_archimedean} to each family.
  Firstly, the Clayton copula has inverse generator
  $\psi^{-1}(x) = (1 + x)^{-1/\theta}$ for $\theta > 0$
  \citep[Example~4.23]{nelsen2006}.
  With $x \in (0, \infty)$,
  \begin{align*}
    \bigl(\psi^{-1}\bigr)'(x)
    &=
    - \frac{1}{\theta} (1 + x)^{-1/\theta - 1}
    < 0,
    &\Psi(x)
    &= \biggl( \frac{1 + x}{1 + d \cdot x} \biggr) ^{1/\theta + 1}
  \end{align*}
  and is decreasing, since
  $\frac{\mathrm d}{\mathrm d x} \frac{1 + x}{1 + d \cdot x}
  = \frac{1 - d}{(1 + d \cdot x)^2} \leq 0$.

  Next, the Frank copula has inverse generator
  $\psi^{-1}(x) = -\frac{1}{\theta}
  \log \bigl(1 - \bigl(1 - e^{-\theta}\bigr) e^{-x}\bigl)$
  for $\theta > 0$
  \citep[Example~4.24]{nelsen2006}.
  With $x \in (0, \infty)$
  and setting $a = \bigl(1 - e^{-\theta}\bigr) \in (0, 1)$,
  \begin{align*}
    \bigl(\psi^{-1}\bigr)'(x)
    &=
    -\frac{1}{\theta}
    \frac{a e^{-x}}{1 - a e^{-x}}
    < 0,
    &\Psi(x)
    &=
    \frac{e^x - a}
    {e^{d x} - a}, \\
    \Psi'(x)
    &= \frac{(1 - d) e^{(d+1) x} - a e^x + a d e^{d x}}
    {\bigl(e^{d x} - a\bigr)^2}
    \leq 0,
  \end{align*}
  since by convexity,
  $e^{d x} \leq \frac{d-1}{d} e^{(d+1)x} + \frac{1}{d} e^x$.
  Thus $\Psi$ is non-increasing.

  Finally, the Gumbel--Hougaard copula has inverse generator
  $\psi^{-1}(x) = \exp\bigl(-x^{1/\theta}\bigr)$
  for $\theta \geq 1$ \citep[Example~4.25]{nelsen2006}.
  With $x \in (0, \infty)$,
  \begin{align*}
    \bigl(\psi^{-1}\bigr)'(x)
    &=
    - \frac{1}{\theta}
    x^{1/\theta - 1}
    \exp \bigl(-x^{1/\theta}\bigr)
    < 0,
    &\Psi(x)
    &= d^{1/\theta - 1}
    \exp \bigl(x^{1/\theta}(1 - d^{1/\theta})\bigr)
  \end{align*}
  and is decreasing, since $d^{1/\theta} \geq 1$.
\end{proof}

\begin{proof}[Details for Example~\ref{ex:mixture_copula}]
  As each $C_k$ is diagonally convex,
  $x \mapsto \sum_{k=1}^K p_k C_k(x, \ldots, x)$
  is convex. The mixture copula is verified to be a copula
  using \cite[Definition~2.10.6]{nelsen2006}.
\end{proof}

\begin{proof}[Details for Example~\ref{ex:convex_gaussian}]
  Let $\mu = 0$ and $\sigma = 1$;
  consider $\mu + \sigma X_i$ for the general case.
  By Theorem~\ref{thm:convex},
  \begin{align*}
    \P \biggl( x < \max_{i \in [d]} X_i \leq x + \varepsilon \biggr)
    &\leq
    \bigl( \Phi(x + \varepsilon) - \Phi(x) \bigr)
    \biggl\{
      \frac{1}{1 - \Phi(x)}
      \land d
    \biggr\}.
  \end{align*}
  For $x \leq 0$ we have $\Phi(x) \leq 1/2$ and
  $\Phi(x + \varepsilon) - \Phi(x) \leq \varepsilon \phi(0)
  = \frac{1}{\sqrt{2\pi}}$, so
  \begin{align*}
    \bigl( \Phi(x + \varepsilon) - \Phi(x) \bigr)
    \biggl\{
      \frac{1}{1 - \Phi(x)}
      \land d
    \biggr\}
    \leq
    \frac{\varepsilon}{\sqrt{2 \pi}}
    (2 \land d)
    = \varepsilon \sqrt{\frac{2}{\pi}}.
  \end{align*}
  If $x \geq 0$ then
  $\Phi(x + \varepsilon) - \Phi(x) \leq \varepsilon \phi(x)$
  and $\frac{\phi(x)}{1 - \Phi(x)} \leq
  \frac{2}{\sqrt{x^2 + 4} - x} \leq x + 1$
  by \cite{birnbaum1942inequality}, so
  \begin{align*}
    \bigl( \Phi(x + \varepsilon) - \Phi(x) \bigr)
    \biggl\{
      \frac{1}{1 - \Phi(x)}
      \land d
    \biggr\}
    \leq
    \varepsilon \phi(x)
    \biggl\{
      \frac{x + 1}{\phi(x)}
      \land d
    \biggr\}
    =
    \varepsilon
    \bigl\{
      (x + 1)
      \land \bigl(d \cdot \phi(x)\bigr)
    \bigr\}.
  \end{align*}
  Setting $x = \sqrt{2 \log d} \geq 0$
  gives $x + 1 = \sqrt{2 \log d} + 1$
  and $d \cdot \phi(x) = \frac{1}{\sqrt{2\pi}}$.
  Since $x \mapsto x + 1$ is continuous and strictly increasing,
  while $x \mapsto d \cdot \phi(x)$ is continuous and strictly decreasing
  on $[0, \infty)$, we have
  \begin{align}
    \nonumber
    \sup_{x \geq 0}
    \bigl\{
      (x + 1)
      \land \bigl(d \cdot \phi(x)\bigr)
    \bigr\}
    &= \inf_{x \geq 0}
    \bigl\{
      (x + 1)
      \lor \bigl(d \cdot \phi(x)\bigr)
    \bigr\} \\
    \label{eq:minimax}
    &\leq
    \Bigl(
      \sqrt{2 \log d} + 1
    \Bigr)
    \lor
    \frac{1}{\sqrt{2\pi}}
    = \sqrt{2 \log d} + 1.
  \end{align}
  The result follows as
  $\sqrt{2 \log d} + 1 \geq \sqrt{\frac{2}{\pi}}$.
\end{proof}

\begin{proof}[Details for Example~\ref{ex:multivariate_gaussian}]
  The first result follows by Lemma~\ref{lem:gaussian_copula_convex}
  and Example~\ref{ex:convex_gaussian}.
  For the second, consider the $2d$-dimensional random vector
  $Y = (X_1 - \mu, \ldots, X_d-\mu, \mu-X_1, \ldots, \mu-X_d)$,
  which has a multivariate Gaussian distribution
  with $Y_i \sim \mathcal N(0, \sigma^2)$ for each $i \in [2d]$.
  By the first inequality,
  \begin{align*}
    \P \biggl( x < \max_{i \in [d]} |X_i - \mu| \leq x + \varepsilon \biggr)
    &=
    \P \biggl( x < \max_{i \in [2d]} Y_i \leq x + \varepsilon \biggr)
    \leq
    \frac{\varepsilon}{\sigma}
    \Bigl( \sqrt{2 \log 2 d} + 1 \Bigr).
    \qedhere
  \end{align*}
\end{proof}

\begin{proof}[Details for Example~\ref{ex:weibull_convex}]
  We assume that the scale parameter is $\lambda = 1$;
  the general result for $\lambda > 0$ then follows by considering
  $\lambda X_i$ for $i \in [d]$.
  The CDF and Lebesgue density of $X_i$ are therefore
  \begin{align*}
    F(x)
    &=
    1 - \exp\bigl(-x^\alpha\bigr),
    &f(x)
    &= \alpha
    x^{\alpha - 1}
    \exp\bigl(-x^\alpha\bigr).
  \end{align*}
  With $x_* = \bigl( \frac{\alpha-1}{\alpha} \bigr)^{1/\alpha}$,
  $f$ increases on $[0, x_*]$
  and decreases on $[x_*, \infty)$.
  By Theorem~\ref{thm:convex}, if $x \leq x_*$,
  \begin{align*}
    \P \biggl( x < \max_{i \in [d]} X_i \leq x + \varepsilon \biggr)
    &\leq
    \frac{\varepsilon f(x_*)}{1 - F(x_*)}
    =
    \varepsilon
    \alpha
    x_*^{\alpha - 1}
    =
    \varepsilon
    \alpha
    \biggl( \frac{\alpha-1}{\alpha} \biggr)^{\frac{\alpha - 1}{\alpha}}
    \leq
    \varepsilon
    \alpha.
  \end{align*}
  Conversely, if $x > x_*$ then
  \begin{align*}
    \P \biggl( x < \max_{i \in [d]} X_i \leq x + \varepsilon \biggr)
    &\leq
    \varepsilon f(x)
    \biggl\{
      \frac{1}{1 - F(x)}
      \land d
    \biggr\}
    =
    \varepsilon
    \alpha x^{\alpha-1}
    \Bigl\{
      1
      \land
      \Bigl(
        d
        \exp\bigl(-x^\alpha\bigr)
      \Bigr)
    \Bigr\}.
  \end{align*}
  Take $x = (\log d + 1)^{1/\alpha} \geq x_*$
  so that $x^{\alpha - 1} = (\log d + 1)^{\frac{\alpha-1}{\alpha}}$
  and $\exp\bigl(-x^\alpha\bigr) \leq \frac{1}{d}$.
  As $x \mapsto x^{\alpha - 1}$ is increasing
  and $x \mapsto x^{\alpha-1} \exp\bigl(-x^\alpha\bigr)$
  is decreasing to zero on $[x_*, \infty)$, we have
  \begin{align*}
    \sup_{x \geq x_*}
    \Bigl\{
      x^{\alpha-1}
      \land
      \Bigl(
        d
        x^{\alpha-1}
        \exp\bigl(-x^\alpha\bigr)
      \Bigr)
    \Bigr\}
    &\leq
    \inf_{x \geq x_*}
    \Bigl\{
      x^{\alpha-1}
      \lor
      \Bigl(
        d
        x^{\alpha-1}
        \exp\bigl(-x^\alpha\bigr)
      \Bigr)
    \Bigr\}
    \leq
    (\log d + 1)^{\frac{\alpha-1}{\alpha}}.
    \qedhere
  \end{align*}
\end{proof}

\begin{proof}[Details for Example~\ref{ex:gumbel_convex}]
  We assume that the scale parameter is $\lambda = 1$;
  the general result for $\lambda > 0$ follows by considering
  $\lambda X_i$ for $i \in [d]$.
  The CDF, Lebesgue density, and hazard function of $X_i$ are
  \begin{align*}
    F(x)
    &=
    1 - \exp\bigl(-e^{x}\bigr),
    &f(x)
    &= \exp\bigl(x-e^{x}\bigr),
    &h(x)
    &= e^x.
  \end{align*}
  Note that $f(x)$ is increasing on $(-\infty, 0]$,
  so for $x \leq 0$, by Theorem~\ref{thm:convex},
  there exists $(X_1, \ldots, X_d)$ with
  \begin{align*}
    \P \biggl( x < \max_{i \in [d]} X_i \leq x + \varepsilon \biggr)
    &\leq
    \frac{\varepsilon f(0)}{1 - F(0)}
    = \varepsilon.
  \end{align*}
  Further, $f$ is decreasing on $[0, \infty)$
  while $h$ is increasing on $[0, \infty)$.
  Note that for $d \geq 2$,
  we have $h(x) = d \cdot f(x)$ if and only if $x = \log \log d$,
  and if $d=1$ then $d \cdot f(x) \leq 1$.
  So for $x \geq 0$,
  \begin{align*}
    \P \biggl( x < \max_{i \in [d]} X_i \leq x + \varepsilon \biggr)
    &\leq
    \varepsilon f(x)
    \biggl\{
      \frac{1}{1 - F(x)}
      \land d
    \biggr\}
    \leq
    \varepsilon
    \bigl\{
      h(x)
      \land \bigl(d \cdot f(x)\bigr)
    \bigr\}
    \leq \varepsilon \bigl(1 + \log d\bigr).
    \qedhere
  \end{align*}
\end{proof}

\begin{proof}[Details for Example~\ref{ex:pareto_convex}]
  We assume that the scale parameter is $\lambda = 1$;
  the general result follows by considering
  $\lambda X_i$ for $i \in [d]$.
  The CDF, Lebesgue density, and hazard function of $X_i$ are therefore
  \begin{align*}
    F(x)
    &=
    1 - x^{-\alpha},
    &f(x)
    &= \alpha
    x^{-\alpha - 1},
    &h(x)
    &= \alpha / x,
  \end{align*}
  for $x \geq 1$. Since $f$ and $h$ are both decreasing,
  and by Theorem~\ref{thm:convex},
  \begin{align*}
    \P \biggl( x < \max_{i \in [d]} X_i \leq x + \varepsilon \biggr)
    &\leq
    \varepsilon
    \bigl\{h(x) \land \bigl(d \cdot f(x)\bigr)\bigr\}
    \leq
    \varepsilon
    \bigl\{h(1) \land \bigl(d \cdot f(1)\bigr)\bigr\}
    = \alpha \varepsilon.
    \qedhere
  \end{align*}
\end{proof}

\begin{proof}[Details for Example~\ref{ex:gamma}]
  We may assume that the scale parameter is $\lambda = 1$.
  Writing $\Gamma(\alpha, x) \vcentcolon=
  \int_x^\infty t^{\alpha-1} e^{-t} \diffi t$
  for $\alpha \geq 1$ and $x \geq 0$,
  the Lebesgue density and distribution function of
  $X_1$ are
  \begin{align*}
    f(x) &=
    \frac{x^{\alpha - 1} e^{-x}}{\Gamma(\alpha)},
    &F(x) &=
    1 - \frac{\Gamma(\alpha, x)}{\Gamma(\alpha)}.
  \end{align*}
  Since $\alpha \geq 1$, we bound
  \begin{align*}
    \Gamma(\alpha, x)
    &= \int_0^\infty (t+x)^{\alpha-1} e^{-t - x} \diffi t
    \geq x^{\alpha-1} e^{-x} \int_0^\infty e^{-t} \diffi t
    = x^{\alpha-1} e^{-x}.
  \end{align*}
  By the mean value theorem, for some $y \in [x, x + \varepsilon]$,
  and since $F(x)$ is increasing,
  \begin{align*}
    \frac{F(x + \varepsilon) - F(x)}{1 - F(x)}
    &=
    \frac{\varepsilon f(y)}{1 - F(x)}
    \leq
    \frac{\varepsilon f(y)}{1 - F(y)}
    = \frac{\varepsilon y^{\alpha - 1} e^{-y}}{\Gamma(\alpha, y)}
    \leq \varepsilon.
  \end{align*}
  Thus, by Theorem~\ref{thm:convex},
  \begin{align*}
    \P \biggl( x < \max_{j \in [d]} X_j \leq x + \varepsilon \biggr)
    &\leq
    \bigl( F(x + \varepsilon) - F(x) \bigr)
    \biggl\{
      \frac{1}{1 - F(x)}
      \land d
    \biggr\}
    \leq
    \varepsilon.
    \qedhere
  \end{align*}
\end{proof}

\begin{proof}[Details for Example~\ref{ex:gmm}]
  We begin by showing that the first term on the right-hand side
  is an upper bound for the left-hand side.
  Note that the common CDF and density function of
  $T_j$, for $j \in [d]$, are
  \begin{align*}
    F(x)
    &=
    \sum_{k=1}^K p_k
    \Phi\biggl(\frac{x}{\sigma_{k}}\biggr),
    &f(x)
    &=
    \sum_{k=1}^K \frac{p_k}{\sigma_k}
    \phi\biggl(\frac{x}{\sigma_{k}}\biggr),
  \end{align*}
  respectively. As $(T_1, \ldots, T_d)$ are i.i.d.,
  and recalling that the independence copula is diagonally convex
  by Example~\ref{ex:frechet-hoeffding-convex},
  we have by Theorem~\ref{thm:convex} that
  \begin{align*}
    &\P \biggl( x < \max_{j \in [d]} T_j \leq x + \varepsilon \biggr)
    \leq
    \bigl( F(x + \varepsilon) - F(x) \bigr)
    \biggl\{
      \frac{1}{1 - F(x)}
      \land d
    \biggr\} \\
    &\quad=
    \sum_{k=1}^K p_k
    \Bigl(
      \Phi(x / \sigma_k + \varepsilon / \sigma_k)
      - \Phi(x / \sigma_k)
    \Bigr)
    \Biggl\{
      \frac{1}
      {1 - \sum_{k=1}^K p_k \Phi(x / \sigma_k)}
      \land d
    \Biggr\} \\
    &\quad\leq
    \sum_{k=1}^K p_k
    \Bigl(
      \Phi(x / \sigma_k + \varepsilon / \sigma_k)
      - \Phi(x / \sigma_k)
    \Bigr)
    \biggl\{
      \frac{1}
      {p_1 \bigl(1 - \Phi(x)\bigr)}
      \land d
    \biggr\}.
  \end{align*}
  Consider first the case that $x \leq 0$;
  then $\Phi(x + \varepsilon) - \Phi(x) \leq \varepsilon \phi(0)$
  and $\Phi(x) \leq 1/2$, so
  \begin{align*}
    \P \biggl( x < \max_{j \in [d]} T_j \leq x + \varepsilon \biggr)
    &\leq
    \frac{2 \varepsilon \phi(0)}{p_1}
    \sum_{k=1}^K
    \frac{p_k}{\sigma_k}.
  \end{align*}
  Next, if $x > 0$ then
  $\Phi(x + \varepsilon) - \Phi(x) \leq \varepsilon \phi(x)$
  and $\frac{\phi(x)}{1 - \Phi(x)} \leq x + 1$
  by \cite{birnbaum1942inequality}, so
  \begin{align*}
    \P \biggl( x < \max_{j \in [d]} T_j \leq x + \varepsilon \biggr)
    &\leq
    \varepsilon
    \sum_{k=1}^K p_k
    \frac{\phi(x / \sigma_k)}{\sigma_k \phi(x)}
    \biggl\{
      \frac{\phi(x)}
      {p_1 \bigl(1 - \Phi(x)\bigr)}
      \land \bigl( d \cdot \phi(x) \bigr)
    \biggr\} \\
    &\leq
    \varepsilon
    \sum_{k=1}^K p_k
    \frac{\phi(x / \sigma_k)}{\sigma_k \phi(x)}
    \biggl\{
      \frac{x + 1}{p_1} \land \bigl( d \cdot \phi(x) \bigr)
    \biggr\}.
  \end{align*}
  If $x \geq 1$ then since
  $0 < \sigma_k \leq 1$ for each $k \in [K]$, it follows that
  $2 \log(1/\sigma_k) \leq 1/\sigma_k^2 - 1$ and so
  we have $x \geq \sqrt{\frac{2 \log (1/\sigma_k)}{1/\sigma_k^2 - 1}}$.
  Then $\log(1/\sigma_k) \leq x^2 (1/\sigma_k^2 - 1)/2$,
  so $e^{-x^2/(2 \sigma_k^2)} \leq \sigma_k e^{-x^2/2}$
  and hence $\phi(x/\sigma_k) \leq \sigma_k \phi(x)$.
  In this case, by \eqref{eq:minimax},
  \begin{align*}
    \P \biggl( x < \max_{j \in [d]} T_j \leq x + \varepsilon \biggr)
    &\leq
    \frac{\varepsilon}{p_1}
    \bigl\{
      (x + 1) \land \bigl( d \cdot \phi(x) \bigr)
    \bigr\}
    \leq
    \frac{\varepsilon}{p_1}
    \Bigl( \sqrt{2 \log d} + 1 \Bigr).
  \end{align*}
  Alternatively, if $0 < x < 1$,
  then since $\sigma_k \leq 1$ for each $k \in [K]$,
  we have $\phi(x/\sigma_k) \leq \phi(x)$ so
  \begin{align*}
    \P \biggl( x < \max_{j \in [d]} T_j \leq x + \varepsilon \biggr)
    &\leq
    \frac{2 \varepsilon}{p_1}
    \sum_{k=1}^K
    \frac{p_k}{\sigma_k}.
  \end{align*}
  Combining these cases, we deduce that for all $x \in \R$,
  \begin{align*}
    \P \biggl( x < \max_{j \in [d]} T_j \leq x + \varepsilon \biggr)
    &\leq
    \biggl\{
      \frac{\varepsilon}{p_1}
      \Bigl( \! \sqrt{2 \log d} + 1 \Bigr)
      \!
    \biggr\}
    \lor
    \Biggl\{
      \frac{2 \varepsilon}{p_1}
      \sum_{k=1}^K
      \frac{p_k}{\sigma_k}
    \Biggr\}
    \leq
    \frac{\varepsilon}{p_1}
    \Biggl(
      \!
      \sqrt{2 \log d} +
      2 \sum_{k=1}^K
      \frac{p_k}{\sigma_k}
    \Biggr).
  \end{align*}
  We now address the second term on the right-hand side.
  By the second equality in \eqref{eq:gmm},
  $(T_1, \ldots, T_d)$ follow a Gaussian mixture distribution
  with $K^d$ components, so consider the following representation:
  \begin{align*}
    (T_1, \ldots, T_d)
    &=
    \sum_{k_1 = 1}^K
    \cdots
    \sum_{k_d = 1}^K
    \I \bigl\{Z = (k_1, \ldots, k_d)\bigr\}
    \bigl(Y_{1, k_1}, \ldots, Y_{d, k_d}\bigr),
  \end{align*}
  where the latent group assignment
  $Z$ takes values in $[K]^d$
  with $\P\bigl(Z = (k_1, \ldots, k_d)\bigr) = \prod_{j=1}^d p_{k_j}$,
  and with $Y_{j, k} \sim \cN\bigl(0, \sigma_{k}^2\bigr)$
  independently for $(j, k) \in [d] \times [K]$ and
  independently of $Z$.
  Therefore by conditioning on $Z$ and applying
  Nazarov's inequality for unequal variances
  \citep[Theorem~1]{chernozhukov2017detailed},
  \begin{align*}
    &\P \biggl( x < \max_{j \in [d]} T_j \leq x + \varepsilon \biggr)
    =
    \E \biggl[
      \P \biggl( x < \max_{j \in [d]} T_j \leq x + \varepsilon
      \Bigm| Z \biggr)
    \biggr] \\
    &\quad=
    \sum_{k_1 = 1}^K
    \cdots
    \sum_{k_d = 1}^K
    \Biggl(
      \prod_{j=1}^d p_{k_j}
    \Biggr)
    \P \biggl( x < \max_{j \in [d]} T_j \leq x + \varepsilon
      \Bigm| Z = (k_1, \ldots, k_d)
    \biggr) \\
    &\quad=
    \sum_{k_1 = 1}^K
    \cdots
    \sum_{k_d = 1}^K
    \Biggl(
      \prod_{j=1}^d p_{k_j}
    \Biggr)
    \P \biggl( x < \max_{j \in [d]} Y_{j, k_j}
      \leq x + \varepsilon
    \biggr) \\
    &\quad\leq
    \sum_{k_1 = 1}^K
    \cdots
    \sum_{k_d = 1}^K
    \Biggl(
      \prod_{j=1}^d p_{k_j}
    \Biggr)
    \frac{\varepsilon}{\min_{j \in [d]} \sigma_{k_j}}
    \Bigl( \sqrt{2 \log d} + 2 \Bigr) \\
    &\quad\leq
    \frac{\varepsilon \bigl( \sqrt{2 \log d} + 2 \bigr)}{\sigma}
    \prod_{j=1}^d
    \Biggl( \sum_{k = 1}^K p_{k} \Biggr)
    =
    \frac{\varepsilon}{\sigma}
    \Bigl( \sqrt{2 \log d} + 2 \Bigr).
  \end{align*}
  We remark that the inequality
  $\min_{j \in [d]} \sigma_{k_j} \geq \sigma$
  is essentially optimal in the regime where
  the dimension $d$ is much larger than
  the number of original components $K$, since
  they differ in only $(K-1)^d$ of the
  $K^d$ possible values for $(k_1, \ldots, k_d)$.
  Further, $\sum_{j=1}^d \I\{k_j = K\}$ is typically around
  $d/K$, so with $\sigma_{k_{(1)}} \leq \cdots \leq \sigma_{k_{(d)}}$,
  the quantity
  $\max_{j \in [d]} \bigl(1 + \sqrt{2 \log j}\bigr) / \sigma_{k_{(j)}}$
  is usually at least on the order
  $\bigl(1 + \sqrt{2 \log (d/K)}\bigr) / \sigma$
  and hence not much smaller than
  $\bigl(1 + \sqrt{2 \log d}\bigr) / \sigma$.
  Therefore the refined version of Nazarov's inequality
  given by \citet[Theorem~10]{deng2020beyond}
  provides no significant improvement in this setting.
\end{proof}
\section{Conclusion}%
\label{sec:conclusion}

We presented sharp upper and lower bounds for the pointwise
concentration function of the
maximum (or minimum) statistic of $d$ identically distributed
random variables, under no further assumptions on their dependence
structure (copula).
When further restricted to copulas with convex diagonal sections,
we demonstrated an improved (and similarly optimal) upper bound
on the aforementioned concentration function.
We verified this condition for a range of popular copulas
and applied our results to several different marginal distributions.
Among other contributions, we recovered a version of Nazarov's inequality
with substantially relaxed assumptions and derived
similar results for non-Gaussian laws.
We presented an application to high-dimensional statistical inference,
giving an explicit example pertaining to Gaussian mixture
approximations for factor models.

There are some potential directions for future research.
Firstly, our main results apply only when the marginal
distributions of each entry in $(X_1, \ldots, X_d)$ agree.
This is a somewhat restrictive assumption, precluding applications
in settings where the random vector of interest is not
standardized entry-wise. For instance,
in Example~\ref{ex:multivariate_gaussian}
we are currently unable to accommodate the setting
of $(X_1, \ldots, X_d) \sim \cN(\mu, \Sigma)$ with
$\mu \in \R^d$ and $\Sigma \in \R^{d \times d}$ an
arbitrary positive semi-definite matrix, which is
handled by Nazarov's inequality
\citep{nazarov2003maximal,chernozhukov2017detailed}
whenever $\min_{i \in [d]} \Sigma_{i i} > 0$.
It is desirable to know whether anti-concentration inequalities
can be derived in regimes where the
marginal laws are Gaussian but the
copula is non-Gaussian, for example, following our
Example~\ref{ex:convex_gaussian}.
The main challenge in establishing such generalizations
is to formulate a natural extension of the ``diagonally convex''
property introduced in Definition~\ref{def:diagonally_convex},
along with a corresponding result analogous to our Theorem~\ref{thm:convex}.
A secondary task would then be to verify the new condition for a
selection of popular multivariate copulas;
initial investigation suggests that this is unlikely to be as
straightforward as in Section~\ref{sec:convex}.

A closely related problem is that of providing bounds for the
probability that $(X_1, \ldots, X_d)$ lies near the perimeter
of a (high-dimensional) rectangle. That is, to control
\begin{align}
  \label{eq:high_dim_rectangle}
  \P \Biggl(
    \bigcap_{i=1}^d
    \bigl\{
      X_i \leq x_i + \varepsilon_i
    \bigr\}
  \Biggr)
  - \P \Biggl(
    \bigcap_{i=1}^d
    \bigl\{
      X_i \leq x_i
    \bigr\}
  \Biggr)
\end{align}
where $x_i \in \R$ and $\varepsilon_i > 0$ for $i \in [d]$.
Setting $Y_i = (X_i - x_i) / \varepsilon_i$
for $i \in [d]$, \eqref{eq:high_dim_rectangle} is equal to
$\P \bigl( 0 < \max_{i \in [d]} Y_i \leq 1 \bigr)$.
To establish bounds for \eqref{eq:high_dim_rectangle},
it would be sufficient to generalize our main results
to the setting where the marginal distributions of
$X_1, \ldots, X_d$ may differ.
See \citet[Lemma~2.2, and references therein]{koike2021notes}
for an example in the multivariate Gaussian setting.

Our anti-concentration results
are concerned with the worst-case choice of a
diagonally convex copula. As such, they do not require knowledge or
estimation of the underlying copula. However, they therefore
may not provide ``dimension-free'' results; that is, bounds which depend
on the dimension $d$ only through quantities relating directly to
the law of the random vector $(X_1, \ldots, X_d)$, such as the
expected maximum statistic
$\E[\max_{i \in [d]} |X_i|]$.
See \citet{kozbur2021dimension} for an example of such a result in the
Gaussian setting;
the development of generalized versions
would be of interest.

We focus on the maximum statistic because its
distribution function admits a simple closed form in terms
of the common marginal law and the diagonal of the copula.
However, it may be possible to extend our results and techniques to
other suitably monotone functions of $(X_1, \ldots, X_d)$.
For example, \cite{kozbur2021dimension}
studied the $k$-max order statistics
for jointly Gaussian variables.
 
\section{Acknowledgments}
We thank Boris Hanin and Grigoris Paouris for their comments.
The authors gratefully acknowledge financial support from the National Science
Foundation through grant DMS-2210561.
 
\bibliographystyle{apalike}
\bibliography{references}

\begin{thebibliography}{}

\bibitem[Aizenman et~al., 2009]{aizenman2009bernoulli}
Aizenman, M., Germinet, F., Klein, A., and Warzel, S. (2009).
\newblock On {Bernoulli} decompositions for random variables, concentration
  bounds, and spectral localization.
\newblock {\em Probability Theory and Related Fields}, 143:219--238.

\bibitem[Arakelian and Karlis, 2014]{arakelian2014clustering}
Arakelian, V. and Karlis, D. (2014).
\newblock Clustering dependencies via mixtures of copulas.
\newblock {\em Communications in Statistics--Simulation and Computation},
  43(7):1644--1661.

\bibitem[Bakshi et~al., 2020]{bakshi2020outlier}
Bakshi, A., Diakonikolas, I., Hopkins, S.~B., Kane, D., Karmalkar, S., and
  Kothari, P.~K. (2020).
\newblock Outlier-robust clustering of {Gaussians} and other non-spherical
  mixtures.
\newblock In {\em 2020 IEEE 61st Annual Symposium on Foundations of Computer
  Science}, pages 149--159. IEEE.

\bibitem[Belloni et~al., 2024]{belloni2024anti}
Belloni, A., Fang, E.~X., and Shen, S. (2024).
\newblock Anti-concentration inequalities for the difference of maxima of
  {Gaussian} random vectors.
\newblock {\em \arxivref{2408.13348}}.

\bibitem[Birnbaum, 1942]{birnbaum1942inequality}
Birnbaum, Z.~W. (1942).
\newblock An inequality for {Mills'} ratio.
\newblock {\em Annals of Mathematical Statistics}, 13:245--246.

\bibitem[Bobkov and Chistyakov, 2015]{bobkov2015}
Bobkov, S.~G. and Chistyakov, G.~P. (2015).
\newblock On concentration functions of random variables.
\newblock {\em Journal of Theoretical Probability}, 28:976--988.

\bibitem[Cattaneo et~al., 2024]{cattaneo2024uniform}
Cattaneo, M.~D., Feng, Y., and Underwood, W.~G. (2024).
\newblock Uniform inference for kernel density estimators with dyadic data.
\newblock {\em Journal of the American Statistical Association},
  119(548):2695--2708.

\bibitem[Cattaneo et~al., 2025]{cattaneo2022yurinskii}
Cattaneo, M.~D., Masini, R.~P., and Underwood, W.~G. (2025).
\newblock {Yurinskii}'s coupling for martingales.
\newblock {\em Annals of Statistics, forthcoming}.

\bibitem[Cattaneo and Yu, 2025]{cattaneo2024strong}
Cattaneo, M.~D. and Yu, R.~R. (2025).
\newblock Strong approximations for empirical processes indexed by {Lipschitz}
  functions.
\newblock {\em Annals of Statistics}, 81(2):667--737.

\bibitem[Chernozhukov et~al., 2013]{chernozhukov2013gaussian}
Chernozhukov, V., Chetverikov, D., and Kato, K. (2013).
\newblock Gaussian approximations and multiplier bootstrap for maxima of sums
  of high-dimensional random vectors.
\newblock {\em Annals of Statistics}, 41(6):2786--2819.

\bibitem[Chernozhukov et~al., 2014a]{chernozhukov2014anti}
Chernozhukov, V., Chetverikov, D., and Kato, K. (2014a).
\newblock Anti-concentration and honest, adaptive confidence bands.
\newblock {\em Annals of Statistics}, 42(5):1787--1818.

\bibitem[Chernozhukov et~al., 2014b]{chernozhukov2014gaussian}
Chernozhukov, V., Chetverikov, D., and Kato, K. (2014b).
\newblock Gaussian approximation of suprema of empirical processes.
\newblock {\em Annals of Statistics}, 42(4):1564--1597.

\bibitem[Chernozhukov et~al., 2015]{chernozhukov2015comparison}
Chernozhukov, V., Chetverikov, D., and Kato, K. (2015).
\newblock Comparison and anti-concentration bounds for maxima of {Gaussian}
  random vectors.
\newblock {\em Probability Theory and Related Fields}, 162(1):47--70.

\bibitem[Chernozhukov et~al., 2017a]{chernozhukov2017central}
Chernozhukov, V., Chetverikov, D., and Kato, K. (2017a).
\newblock Central limit theorems and bootstrap in high dimensions.
\newblock {\em Annals of Probability}, 45(4):2309--2352.

\bibitem[Chernozhukov et~al., 2017b]{chernozhukov2017detailed}
Chernozhukov, V., Chetverikov, D., and Kato, K. (2017b).
\newblock Detailed proof of {Nazarov}'s inequality.
\newblock {\em \arxivref{1711.10696}}.

\bibitem[Cs{\"o}rg{\"o} and R{\'e}v{\'e}sz, 1981]{csorgo1981strong}
Cs{\"o}rg{\"o}, M. and R{\'e}v{\'e}sz, P. (1981).
\newblock {\em Strong Approximations in Probability and Statistics}.
\newblock Probability and Mathematical Statistics: a series of monographs and
  textbooks. Academic Press.

\bibitem[Cuculescu and Theodorescu, 2001]{cuculescu2001copulas}
Cuculescu, I. and Theodorescu, R. (2001).
\newblock Copulas: Diagonals, tracks.
\newblock {\em Revue Roumaine de Math{\'e}matiques Pures et Appliqu{\'e}es},
  46(6):731--742.

\bibitem[Deng and Zhang, 2020]{deng2020beyond}
Deng, H. and Zhang, C.-H. (2020).
\newblock Beyond {Gaussian} approximation: bootstrap for maxima of sums of
  independent random vectors.
\newblock {\em Annals of Statistics}, 48(6):3643--3671.

\bibitem[D{\"o}bler and Peccati, 2018]{dobler2018gamma}
D{\"o}bler, C. and Peccati, G. (2018).
\newblock The gamma {Stein} equation and noncentral de {Jong} theorems.
\newblock {\em Bernoulli}, 24(4B):3384--3421.

\bibitem[Durante and Sempi, 2016]{durante2016principles}
Durante, F. and Sempi, C. (2016).
\newblock {\em Principles of Copula Theory}.
\newblock Chapman and Hall/CRC Press, New York.

\bibitem[Embrechts et~al., 2013]{embrechts2013modelling}
Embrechts, P., Kl{\"u}ppelberg, C., and Mikosch, T. (2013).
\newblock {\em Modelling Extremal Events: for Insurance and Finance}, volume~33
  of {\em Stochastic Modelling and Applied Probability}.
\newblock Springer Science \& Business Media.

\bibitem[Fern{\'a}ndez-S{\'a}nchez and {\'U}beda-Flores,
  2018]{fernandez2018constructions}
Fern{\'a}ndez-S{\'a}nchez, J. and {\'U}beda-Flores, M. (2018).
\newblock Constructions of copulas with given diagonal (and opposite diagonal)
  sections and some generalizations.
\newblock {\em Dependence Modeling}, 6(1):139--155.

\bibitem[Fox et~al., 2021]{fox2021combinatorial}
Fox, J., Kwan, M., and Sauermann, L. (2021).
\newblock Combinatorial anti-concentration inequalities, with applications.
\newblock In {\em Mathematical Proceedings of the Cambridge Philosophical
  Society}, volume 171, pages 227--248. Cambridge University Press.

\bibitem[Frank et~al., 1987]{frank1987best}
Frank, M.~J., Nelsen, R.~B., and Schweizer, B. (1987).
\newblock Best-possible bounds for the distribution of a sum---a problem of
  {Kolmogorov}.
\newblock {\em Probability Theory and Related Fields}, 74(2):199--211.

\bibitem[Gaunt et~al., 2017]{gaunt2017stein}
Gaunt, R.~E., Pickett, A.~M., and Reinert, G. (2017).
\newblock Chi-square approximation by {Stein's} method with application to
  {Pearson's} statistic.
\newblock {\em Annals of Applied Probability}, 27:720--756.

\bibitem[Giessing, 2023]{giessing2023}
Giessing, A. (2023).
\newblock Anti-concentration of suprema of {Gaussian} processes and {Gaussian}
  order statistics.
\newblock {\em \arxivref{2310.12119}}.

\bibitem[G{\"o}tze et~al., 2019]{gotze2017large}
G{\"o}tze, F., Naumov, A., Spokoiny, V.~G., and Ulyanov, V.~V. (2019).
\newblock Large ball probabilities, {Gaussian} comparison and
  anti-concentration.
\newblock {\em Bernoulli}, 25(4A):2538--2563.

\bibitem[Jaworski, 2009]{jaworski2009copulas}
Jaworski, P. (2009).
\newblock On copulas and their diagonals.
\newblock {\em Information Sciences}, 179(17):2863--2871.

\bibitem[Koike, 2021]{koike2021notes}
Koike, Y. (2021).
\newblock Notes on the dimension dependence in high-dimensional central limit
  theorems for hyperrectangles.
\newblock {\em Japanese Journal of Statistics and Data Science}, 4:257--297.

\bibitem[Kozbur, 2021]{kozbur2021dimension}
Kozbur, D. (2021).
\newblock Dimension-free anticoncentration bounds for {Gaussian} order
  statistics with discussion of applications to multiple testing.
\newblock {\em \arxivref{2107.10766}}.

\bibitem[Krishnapur, 2016]{krishnapur2016anti}
Krishnapur, M. (2016).
\newblock Anti-concentration inequalities.
\newblock Lecture notes, Advanced Training in Mathematics Workshop in Applied
  Probability, Indian Institute of Technology Bombay.

\bibitem[Kuchibhotla et~al., 2021]{kuchibhotla2021high}
Kuchibhotla, A.~K., Mukherjee, S., and Banerjee, D. (2021).
\newblock High-dimensional {CLT}: Improvements, non-uniform extensions and
  large deviations.
\newblock {\em Bernoulli}, 27(1):192 -- 217.

\bibitem[L{\'e}vy, 1954]{levy1954}
L{\'e}vy, P. (1954).
\newblock {\em Th{\'e}orie de l'Addition Des Variables Al{\'e}atoires}.
\newblock Gauthier-Villars, Paris.

\bibitem[Lindvall, 1992]{Lindvall_1992_Book}
Lindvall, T. (1992).
\newblock {\em Lectures on the Coupling Method}.
\newblock Dover Publications, New York.

\bibitem[Litvak et~al., 2017]{litvak2017adjacency}
Litvak, A.~E., Lytova, A., Tikhomirov, K., Tomczak-Jaegermann, N., and Youssef,
  P. (2017).
\newblock Adjacency matrices of random digraphs: singularity and
  anti-concentration.
\newblock {\em Journal of Mathematical Analysis and Applications},
  445(2):1447--1491.

\bibitem[Livshyts, 2014]{livshyts2014maximal}
Livshyts, G.~V. (2014).
\newblock Maximal surface area of a convex set in $\mathbb{R}^n$ with respect
  to log concave rotation invariant measures.
\newblock In {\em Geometric Aspects of Functional Analysis: Israel Seminar
  (GAFA) 2011--2013}, pages 355--383. Springer.

\bibitem[Livshyts, 2021]{livshyts2021some}
Livshyts, G.~V. (2021).
\newblock Some remarks about the maximal perimeter of convex sets with respect
  to probability measures.
\newblock {\em Communications in Contemporary Mathematics}, 23(05).

\bibitem[Lopes et~al., 2020]{lopes2020bootstrapping}
Lopes, M.~E., Lin, Z., and M{\"u}ller, H.-G. (2020).
\newblock Bootstrapping max statistics in high dimensions: Near-parametric
  rates under weak variance decay and application to functional and multinomial
  data.
\newblock {\em Annals of Statistics}, 48(2):1214--1229.

\bibitem[Meka et~al., 2015]{meka2015anticoncentration}
Meka, R., Nguyen, O., and Vu, V.~H. (2015).
\newblock Anti-concentration for polynomials of independent random variables.
\newblock {\em Theory of Computing}, 12:1--17.

\bibitem[Miller and Samko, 2001]{miller2001completely}
Miller, K.~S. and Samko, S.~G. (2001).
\newblock Completely monotonic functions.
\newblock {\em Integral Transforms and Special Functions}, 12(4):389--402.

\bibitem[Nazarov, 2003]{nazarov2003maximal}
Nazarov, F. (2003).
\newblock On the maximal perimeter of a convex set in $\mathbb{R}^n$ with
  respect to a {Gaussian} measure.
\newblock In {\em Geometric Aspects of Functional Analysis: Israel Seminar,
  2001--2002}, pages 169--187. Springer.

\bibitem[Nelsen, 2006]{nelsen2006}
Nelsen, R.~B. (2006).
\newblock {\em An Introduction to Copulas}.
\newblock Springer Series in Statistics. Springer, New York.

\bibitem[Nie, 2022]{nie2022matrix}
Nie, Z. (2022).
\newblock Matrix anti-concentration inequalities with applications.
\newblock In {\em Proceedings of the 54th Annual ACM SIGACT Symposium on Theory
  of Computing}, pages 568--581.

\bibitem[Paouris, 2012]{paouris2012small}
Paouris, G. (2012).
\newblock Small ball probability estimates for log-concave measures.
\newblock {\em Transactions of the American Mathematical Society},
  364(1):287--308.

\bibitem[Paouris and Valettas, 2018]{paouris2018gaussian}
Paouris, G. and Valettas, P. (2018).
\newblock A {Gaussian} small deviation inequality for convex functions.
\newblock {\em Annals of Probability}, 46(3):1441--1454.

\bibitem[Pollard, 2002]{pollard2002user}
Pollard, D. (2002).
\newblock {\em A User's Guide to Measure Theoretic Probability}.
\newblock Cambridge Series in Statistical and Probabilistic Mathematics.
  Cambridge University Press.

\bibitem[Rudelson and Vershynin, 2015]{rudelson2015small}
Rudelson, M. and Vershynin, R. (2015).
\newblock Small ball probabilities for linear images of high-dimensional
  distributions.
\newblock {\em International Mathematics Research Notices},
  2015(19):9594--9617.

\bibitem[Saumard and Wellner, 2014]{saumard_wellner}
Saumard, A. and Wellner, J.~A. (2014).
\newblock {Log-concavity and strong log-concavity: A review}.
\newblock {\em Statistics Surveys}, 8:45--114.

\bibitem[Vershynin and Rudelson, 2007]{vershynin2007}
Vershynin, R. and Rudelson, M. (2007).
\newblock Anti-concentration inequalities.
\newblock In {\em Phenomena in High Dimensions, Third Annual Conference, Samos,
  Greece}.

\end{thebibliography}

\pagebreak

\end{document}